\documentclass[10pt]{article}
\usepackage[english]{babel}
\usepackage{amssymb, amsmath,amsthm, verbatim, cancel, fancybox, graphicx, animate, enumerate, enumitem, amstext, array, cleveref, colonequals, ifthen, booktabs, float, afterpage, caption, subcaption, multirow, tikz-cd, blindtext, setspace, layout, url}
\usepackage[top=1in, bottom=1in, left=1.5in, right=1in]{geometry}

\newcommand{\mysection}[1]{\newpage \vspace*{0.5in} \section{#1}}
\newtheorem{thm}{Theorem}[section]
\newtheorem{lem}[thm]{Lemma}
\newtheorem{prop}[thm]{Proposition}
\newtheorem{cor}[thm]{Corollary}
\newtheorem{sch}[thm]{Scholium}

\theoremstyle{remark}
\newtheorem{rmk}[thm]{Remark}
\theoremstyle{definition}
\newtheorem{defn}[thm]{Definition}

\newcommand{\ds}{\displaystyle}
\newcommand{\x}{\times}
\let\temp\phi
\let\phi\varphi
\let\varphi\temp
\newcommand{\p}[2]{\ds \hat{\phi}_{(#1,#2)}}

\newcommand{\ph}{\ds \hat{\phi}}

\newcommand{\urlb}[1]{\textcolor{blue}{\url{#1}}}

\newcommand{\strutdepth}{\dp\strutbox}
\newcommand{\marginalnote}[1]
   {\strut\vadjust{\kern-\strutdepth\domarginalnote{#1}}}
\newcommand{\domarginalnote}[1]{\vtop to \strutdepth{
  \baselineskip\strutdepth
   \vss\llap{ #1\ \ }\null}}  
\newcounter{showlabelflag}
\newcounter{makelabelflag}

\newcommand{\makelabels}{\setcounter{makelabelflag}{1}}
\newcommand{\hidelabels}{\setcounter{showlabelflag}{2}}
\newcommand{\mylabel}[1]{
  \ifthenelse{\value{makelabelflag}=1}
    {\label{#1}}{}
  \ifthenelse{\value{showlabelflag}=1}
    {\marginpar{#1}}{}\relax}

\newcommand{\figw}{\centering
   \begin{subfigure}{0.45\textwidth}
     \includegraphics[width=\linewidth]{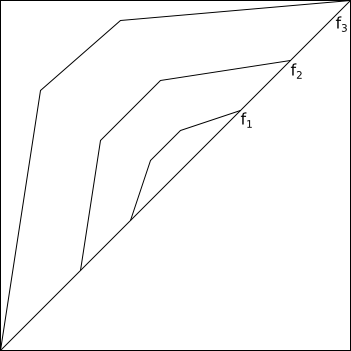}
     \caption{Generators of a standard representation}\mylabel{plw3}
   \end{subfigure}
\hspace{.3cm}
   \begin{subfigure}{0.45\textwidth}
     \includegraphics[width=\linewidth]{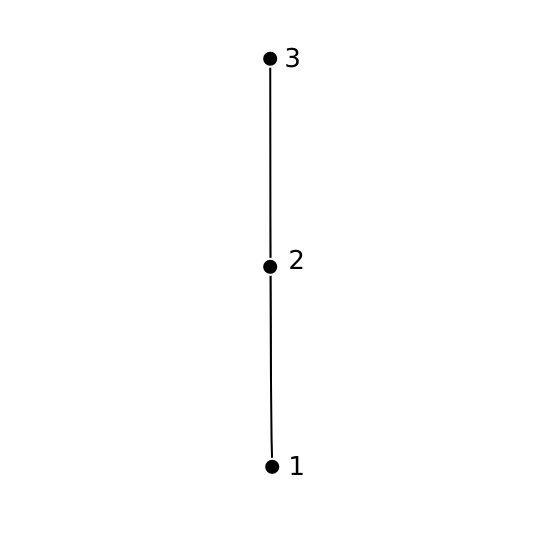}
     \caption{The associated Hasse diagram of orbitals of generators}\mylabel{hw3}
   \end{subfigure}}

\newcommand{\figsplit}{\centering
   \begin{subfigure}{0.45\textwidth}
     \includegraphics[width=\linewidth]{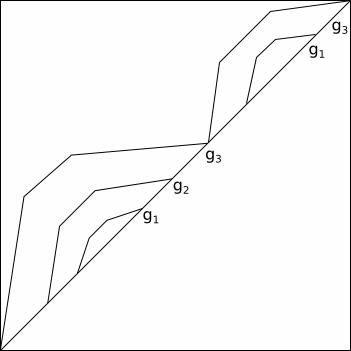}
     \caption{Second maximal tower is split when overlayed with the first}\mylabel{plw3_split}
   \end{subfigure}
\hspace{.3cm}
   \begin{subfigure}{0.45\textwidth}
     \includegraphics[width=\linewidth]{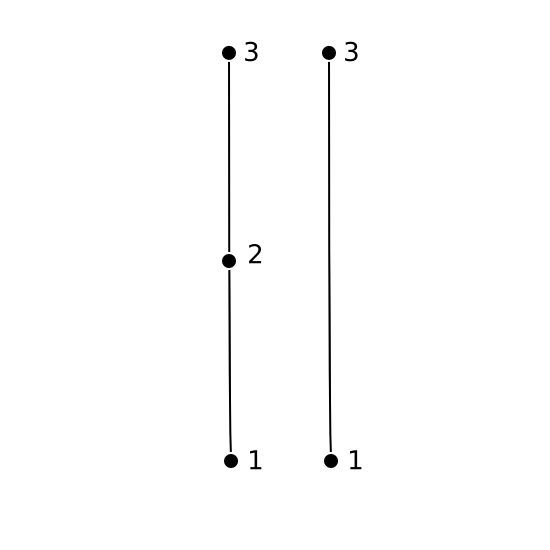}
     \caption{The associated Hasse diagram of orbitals}\mylabel{hw3_split}
   \end{subfigure}}
   
   \newcommand{\purityandnonpurity}{\begin{figure}
\begin{center}
\includegraphics[width=6cm, keepaspectratio]{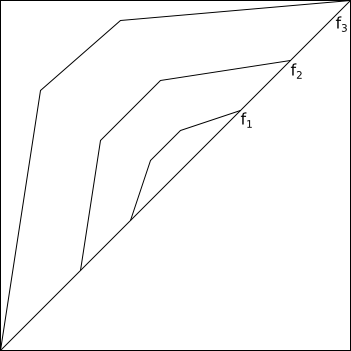}
\caption{A pure tower}\mylabel{purity}
\includegraphics[width=6cm, keepaspectratio]{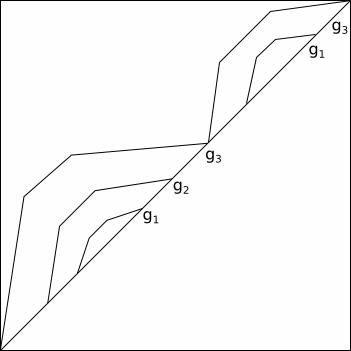}
\caption{Some nonpure towers}\mylabel{nonpurity}
\end{center}
\end{figure}}

 \makelabels
\hidelabels

\title{\vspace*{0.5in}LOCALLY SOLVABLE SUBGROUPS OF PLo(I)}
\author{\\ \\ \\    \\     \\ \\ \\ \\      \\ \\ \\ BY \\ \\ AMANDA TAYLOR \\ \\ BA, University of Maine at Farmington, 2006 \\ MA, Binghamton University, 2009 \\ \\ \\    \\ \\ \\ \\     \\ \\ \\ \\     \\ \\ \\DISSERTATION \\ \\ Submitted in partial fulfillment of the requirements for \\ the degree of Doctor of Philosophy in Mathematical Sciences \\ in the Graduate School of \\ Binghamton University \\ State University of New York \\ 2017}
\date{}

\begin{document}

\pagenumbering{roman}
\maketitle \thispagestyle{empty}

\newpage
\thispagestyle{empty}

\vbox to 8.3truein{}

\centerline{\copyright\ Copyright by Amanda Lee Taylor 2017}
\centerline{}
\centerline{All Rights Reserved}

\newpage

{\baselineskip = 10pt

\vbox to 3.5truein{}

\centerline{Accepted in partial fulfillment of the requirements for}
\centerline{the degree of Doctor of Philosophy in Mathematical Sciences}
\centerline{in the Graduate School of}
\centerline{Binghamton University}
\centerline{State University of New York}
\centerline{2017}
\vskip 105pt

\centerline{July 24, 2017}
\vskip 105pt

\centerline{Matthew G. Brin, Chair and Faculty Advisor}
\centerline{Department of Mathematical Sciences, Binghamton University}
\vskip 10pt

\centerline{Ross Geoghegan, Member}
\centerline{Department of Mathematical Sciences, Binghamton University}
\vskip 10pt

\centerline{Fernando Guzm\'an, Member}
\centerline{Department of Mathematical Sciences, Binghamton University}
\vskip 10pt

\centerline{Justin Moore, Outside Examiner}
\centerline{Department of Mathematics, Cornell University}

}

\newpage

\doublespacing

\centerline{\bf Abstract}
We show that locally solvable subgroups of PLo(I) are countable.  Then for each countable ordered set, we construct a locally solvable subgroup of Thompson's Group F.  We develop machinery for understanding embeddings from solvable subgroups into solvable subgroups.  Finally, we apply this machinery to show the ordered sets used in our construction are invariant under isomorphisms between the groups constructed.  Therefore, we effectively distinguish the groups and provide uncountably many non-isomorphic locally solvable, hence elementary amenable, subgroups of Thompson's Group F.

\newpage

\vspace*{3in}
\noindent{\it To my sister, Stacy, who loved me when no one else did and to my friend, Katrina, who passed away in December 2016.}

\newpage

\centerline{\bf Acknowledgements}
Thank you, Frank Underkuffler, for teaching an undergraduate philosophy course that inspired me, as an art student, to pursue physics and mathematics.  You helped set me on an intellectual path to make meaningful connections which inspire me to this day.  Your unhappiness to learn I was pursuing undergraduate math research instead of art seems  laughable in retrospect.  But sometimes even we don't understand the implications, effects, and importance of our own work, and that's okay.  What's most important is that you cared.

Thank you, Paul Gies, for creating a student-centered proofs-based course focused on non-Euclidean geometry that transformed my experience and appreciation of mathematics and convinced me that I wanted to study math for my undergraduate degree.   I would also like to thank you for seeing potential in me and for your unwavering support that I should pursue a graduate degree, even when I didn't believe it myself.

To my mom, who worked for years as a secretary in a finance department and was promoted to accountant because she was doing the work for it already, in spite of never obtaining any degree:  Thank you for teaching me that persistence and hard work are one of the most important factors in success.  And thank you for being a voracious reader when you weren't too tired from work.

To my dad who loves to make people laugh, to try new things, and to garden: Thanks for helping me appreciate the little things a little more.

To Fernando Guzm\'an:  Thank you for creating graduate classes that conveyed your expertise in ways that were inspiring and empowering.  Thank you also for supporting and challenging me.  Professors can have a huge impact on their students by reaching out in small ways; you are very good at that, and it makes a huge difference.

And last but certainly not least: Thank you Matthew G. Brin for being my Ph.D. advisor, for your humility and genuinely deep love of mathematics, for your sense of humor, your insights into necessary skills for becoming a better researcher, and for your crucial role and support in completing this work.

\newpage
\tableofcontents

\mysection{Introduction}\mylabel{intro}
\pagenumbering{arabic}

Thompson's group F was discovered in 1965 by Richard Thompson in connection with his work in logic. It is the group of all piecewise linear orientation-preserving homeomorphisms of the unit interval with finitely many points of non-differentiability ({\bf breakpoints}) occurring at dyadic rationals and whose slopes are integral powers of 2.  The binary operation is function composition.  Remarkably, the group is two-generated by the functions whose graphs are below.

\begin{center}
\includegraphics[width=10cm, keepaspectratio]{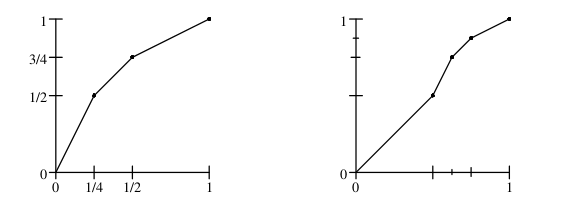}\tiny{image credit: \cite{belk}}
\end{center}

There are several results about F which have garnered interest in it.  F is torsion-free but has no free subgroups of rank 2 or higher, it is lawless, it is finitely presented, it has a beautiful infinite presentation, it is the universal conjugacy idempotent, and it has close ties with associativity (see \cite{nf2}, \cite{cfp}, \cite{fguz}, \cite{bgeo}).  The many intriguing properties of F mean if one thinks something may be true for infinite groups, F is a great test case.  There are also many different representations of F which aid in understanding and calculating.  Problems in F are challenging, too, which makes them fun. 

F has many generalizations which also have interesting properties. However, the subgroup structure of F and these groups is poorly understood, and there are many open questions.  The solvable subgroups were classified by Collin Bleak in papers which came out of his 2005 thesis (see \cite{geo}, \cite{alg}, \cite{min}).  

The current dissertation is an exploration of a more general class of groups called the locally solvable subgroups of Thompson's Group F.  A group is {\bf locally solvable} if every finitely generated subgroup is solvable.  Facts about F are often true of the larger group PLo(I).  This applies to the results of this thesis and they will all be stated about PLo(I).  {\bf PLo(I)} is the group of piecewise linear orientation-preserving homeomorphisms of the unit interval with finitely many breakpoints. 

After proving locally solvable subgroups of PLo(I) are countable, we develop presentations, normal forms, and information about representations of finitely iterated wreath products with $\mathbb{Z}$.  Next we use these results to construct a locally solvable subgroup of PLo(I) for each countable ordered set.  Finally, we develop machinery regarding embeddings between finitely iterated wreath products with $\mathbb{Z}$ which is used to prove the constructed groups are non-isomorphic when the ordered sets are distinct.  A simple consequence is that there are uncountably many non-isomorphic locally solvable subgroups of Thompson's Group F.

The work of this paper began in 2011.  At that time, the author became interested in conjectures surrounding elementary amenable subgroups presented in \cite{ea} and thought that results regarding locally solvable subgroups could be of some use in a classification of these groups.  Furthermore, the author thought the key to understanding the nature of the difficult and elusive question of amenability of $F$ may be through further classification of subgroup structure. 

Since that time, new results have supported the author's vision. In February 2012, the author showed that locally solvable subgroups of PLo(I) are countable and presented these results in conferences in 2012 and 2013.    Since then, other researchers in the field have taken a renewed interest in classifying subgroup structure and proven beautiful and remarkable results (see \cite{markedgroups}, \cite{loc solv}, \cite{chaingroups}).

The history of research in groups of piecewise linear homeomorphisms also indicates there is value in relating geometric properties on pairs of elements to global algebraic properties of groups.  Bleak's results about solvable subgroups are of this nature.  This kind of geometric-to-algebraic study also appeared in earlier papers of Matthew G. Brin such as \cite{ubiquity} and \cite{ea}.  We make heavy use of geometric properties of pairs of elements throughout this thesis.

\subsection{Preliminary Definitions and Background}

Our group actions are right actions, so we use the notation $(x)fg$ to mean $((x)f)g$.  Also, for conjugation and commutators, $f^g := g^{-1}fg$ and $[f,g] := f^{-1}g^{-1}fg$.

Given a function $f \in PLo(I)$, $Supp(f) := \{x \in I \, | \, (x)f \neq x \}$ which we call the {\bf support} of $f$.  Similarly, given a subgroup $G$ in PLo(I), the {\bf support} of $G$ is $Supp(G) := \{x \in I \, | \, (x)f \neq x \text{ for some } f \in G\}$.  An {\bf orbital} of a function $f$ in $PLo(I)$ is a maximal interval contained in $Supp(f)$. Orbitals are open intervals in $\mathbb{R}$, and every function in PLo(I) has only finitely many orbitals because they have only finitely many breakpoints.  An orbital of a subgroup $G$ in PLo(I) is a maximal open interval contained in $Supp(G)$. Note that an orbital of $G$ is a union of orbitals of elements of $G$.    To emphasize whether we are talking about an orbital of a function or an orbital of a group, we sometimes refer to an orbital as an {\bf element orbital} or {\bf group orbital}, respectively.  A {\bf bump} of $f$ is the restriction of $f$ to a single orbital.  A function is called a {\bf one-bump function} if it has only one orbital. 

A {\bf signed orbital} is a pair $(A,f)$ where $f \in PLo(I)$ and $A$ is an orbital of $f$.   We call $A$ the {\bf orbital} of $(A,f)$ and $f$ the {\bf signature}.  Signed orbitals are simply symbols used to represent bumps of functions, and they are clearly in one-to-one correspondence with bumps of functions.  Thus we move freely between the two concepts.

For any subset $C$ of elements of $PLo(I)$, let $\mathcal{SO}(C)$ be the set of all signed orbitals of $C$ and $\mathcal{O}(C)$ the set of all the orbitals of elements in 
$\mathcal{SO}(C)$.  For a set with a single element $g \in$ PLo(I), we will omit set brackets in the previous notations and write only $\mathcal{SO}(g)$ or $\mathcal{O}(g)$. If $C$ is a collection of signed orbitals of PLo(I), we the notations $\mathcal{O}(C)$ and $\mathcal{S}(C)$ to denote the set of all orbitals of $C$ and the set of all signatures of $C$, respectively.   Note that $\mathcal{O}(G)$ refers to the set of all {\bf element orbitals} of elements of $G$, not to group orbitals.

We call a signed orbital $(A,f)$ {\bf pure} if $f$ is a one-bump function.  We call a set $S$ of signed orbitals {\bf pure} if the signatures of $S$ are all one-bump functions.  A one-bump function may also be referred to as {\bf pure}, and a set of functions is {\bf pure} if each element of the set is pure.  

 A {\bf fundamental domain} of a signed orbital $(A,f)$ is a half open interval $[x,xf)$ where $x$ is a point in the orbital $A$. A {\bf fundamental domain} of $f$ is a union of half open intervals---one fundamental domain for each signed orbital of $f$. A set $S$ of signed orbitals is {\bf fundamental} if whenever $(A, f), (B, g) \in S$ are such that $A \subset B$, then $A$ is in a fundamental domain of $(B, g)$.\\

Any collection of orbitals is partially ordered by inclusion. Thus $\mathcal{O}(G)$ is a poset, and we will use some poset notation throughout this paper.  Let $(P, \leq)$ be a partially ordered set and let $A \in P$.  We write $\downarrow A = \{B \, | \, B < A \}$ and refer to this as the {\bf downset of $A$ in $P$}. Similarly, the {\bf upset of $A$ in $P$} is $\uparrow A = \{B | A < B \}$.  Note these sets do not contain $A$.  Also, there is no reference to $P$ in the notation $\downarrow A$ or $\uparrow A$, so the poset will be specified or clear from the context. If $S \subseteq P$, then define $\ds \downarrow S = (\bigcup_{A \in S} \downarrow A) - S$. Note $(\bigcup_{A \in S} \downarrow A) - S = \bigcup_{A \in S}( \downarrow A - S)$.  Define $\uparrow S$ similarly.

An ordering on the poset $\mathcal{SO}(G)$ is inherited from $\mathcal{O}(G)$ by making new chains for distinct elements with the same orbitals.  More explicitly, order $\mathcal{S}(G)$ by the trivial partial order on $G$, that is, if $a \neq b$ in $G$, $a$ and $b$ are incomparable.  Let $P_2$ be this partial order and $P_1$ be the partial order on $\mathcal{O}(G)$. Define the {\bf lexicographical partial order} on a product $A \times B$ of partially ordered sets by  $(a,b) \leq (c,d)$   if and only if $a < c$ or $(a=c \text{ and } b \leq d)$. Extend $P_1$ to a partial order on $\mathcal{SO}(G)$ by taking the product $P = P_1 \times P_2$ with the lexicographical partial ordering.  We use this partial order on $\mathcal{SO}(G)$ throughout this thesis.

We call a chain in the poset $\mathcal{O}(G)$ a {\bf stack} in $G$.  A subset of a stack is also a stack, and hence is called a {\bf substack}.  A {\bf tower} is a chain of orbitals together with an assignment of a signature to each orbital. Hence a tower is set of signed orbitals which is naturally order isomorphic to its underlying stack.  A tower can also be described as a chain in $\mathcal{SO}(G)$ when equipped with partial order $P$ in the last paragraph.  Every subset of a tower is also a tower, so is called a {\bf subtower}.  For convenience, whenever we list the elements of a tower, we will list them from smallest to largest unless otherwise specified.  We often work with totally ordered sets and refer to them simply as {\bf ordered sets}. Whenever a different type of ordering is used, it will be specified or clear from the context.  

Geometrically, one can think of a tower as bumps with nested orbitals.  However, it is possible that signatures of a tower have multiple bumps which create more complicated dynamics.  Since a tower is a set of signed orbitals, our prior definitions for pure and fundamental sets of signed orbitals apply to towers also.  See Figures \ref{purity} and \ref{nonpurity} for examples of pure and non-pure towers. The nonpure situation can get much more complex than pictured.  Since pure towers are simpler to understand, often our arguments build-up from pure to nonpure situations. 

\purityandnonpurity

 Towers were introduced in Collin Bleak's Ph.D. thesis to help classify solvable subgroups of PLo(I).  The results of his thesis were mostly summed up in 3 published papers \cite{geo}, \cite{alg}, and \cite{min}.  The main results of \cite{geo} and \cite{alg} will be of great use to us here.  

Define the {\bf depth} of $G$ denoted depth(G) to be the supremum of all cardinalities of towers in $G$.  The main result of \cite{geo} states

\begin{thm}\mylabel{thm geo}
 A subgroup $G$ of $PLo(I)$ is solvable with derived length $n$ if and only if $depth(G) = n$.
\end{thm}

The power of this theorem is that it allows us to make inductive arguments on towers in solvable groups.  When $G$ is solvable, we may refer to the cardinality of a tower $T \in \mathcal{SO}(G)$ as the {\bf height} of $T$.  The following terminology also appeared in \cite{geo} and is of central concern to this dissertation.

\begin{defn}
 A {\bf transition chain} is a pair of signed orbitals $(A,f)$, $(B,g)$ such that $A \cap B \neq \emptyset$, $A \not\subset B$, and $B \not\subset A$.  We say $G$ has {\bf no transition chains} if no pair of elements of $\mathcal{SO}(G)$ is a transition chain.  
\end{defn}

Visually, a transition chain is a pair of bumps which overlap, but neither properly contains the other.  When transitions chains are not allowed, group actions on the interval are much simpler.  

Let $\mathcal{T}$ be the collection of all groups in $PL_o(I)$ without transition chains.  Thanks to Theorem 3.2 in \cite{loc solv}, we can work geometrically with locally solvable subgroups of PLo(I):

\begin{thm}\mylabel{loc solv}
A subgroup of $PLo(I)$ is locally solvable if and only if it has no transition chains.
\end{thm}

The two main results of this dissertation are

\begin{thm}\mylabel{countableloc}
A locally solvable subgroup of PLo(I) is countable.
\end{thm}

\begin{thm}\mylabel{main}
For each countable ordered set $C$, there exists a locally solvable subgroup $W_C$ of Thompson's Group F which has generating tower $T$ that is order isomorphic to $C$.  Furthermore, $W_C$ is isomorphic to $W_D$ if and only if $C$ is order isomorphic to $D$.
\end{thm}

We prove Theorem \ref{main} by developing geometric ways of viewing embeddings and relations.  The interplay between relations and geometric relationships between orbitals of elements will play a crucial role in our results.

  For our purposes, we need more structure than the cardinality of a tower, so we define a notion called the {\bf type} of a tower.

\begin{defn}
 If $T$ is a tower in $G$, then the {\bf type} of the tower T denoted type(T) is the order type of the ordered set of right endpoints of the orbitals in $\mathcal{O}(T)$.  
\end{defn}

Types of towers give a wealth of information about underlying wreath product structures in the group.  In fact, there is some rigidity of towers under injective homomorphisms which allows us to show the last part of Theorem \ref{main}.  A corollary is

\begin{cor}\mylabel{uncountable}
 There are uncountably many isomorphism classes of groups in $\mathcal{T}$.
\end{cor}

\begin{proof}
There are uncountably many countable ordinals \cite{ordinals}.  Thus there are uncountably many countable ordered sets.
\end{proof}

\sloppy Our main groups of interest  can be built by starting with $\mathbb{Z}$ and forming repeated direct summands, wreath products, and direct limits.  By a wreath product, we mean the {\bf standard restricted wreath product}: The group {\bf A wreath B} denoted $A \wr B$ is $\bigoplus_{i \in B} A \rtimes B$ where $B$ acts on the indexing set of the direct sum by right multiplication.  The group $\bigoplus_{i \in B} A$ is called the {\bf base group} of $A \wr B$.

Consider the following collection of groups indexed by ordinals.
$$W_0 = 1 \text{, } W_\alpha = 
\begin{cases}
W_{\alpha -1} \wr \mathbb{Z}, & \text{if } \alpha \text{ is not a limit ordinal} \\
\ds \lim_{ \beta < \alpha} W_\beta, & \text{if }\alpha \text{ is a limit ordinal}
\end{cases}$$
where $\wr$ is the restricted wreath product and $lim$ is the direct limit using the embedding of $W_{\beta-1}$ in $W_\beta$ as a summand of the base group of $W_\beta$.  We show that if $W_\alpha$ embeds in PLo(I), then $\alpha$ is countable.  Furthermore, we generalize the groups $W_\alpha$ to each countable ordered set $C$, show they embed in $F$, and we use our geometric representations to show they are isomorphic if and only if the underlying ordered sets are isomorphic.  This proves Theorem \ref{main}.

To prove these results, we require more basic terminology and facts.  If $A = (x,y)$ is an orbital of a function $f \in$ PLo(I), we call $x$ and $y$ the {\bf ends} of $A$.  We say something happens {\bf near an end} or {\bf near the ends} of $A$ if it true on some interval $(x,a) \subseteq A$ or true on 2 intervals $(x,a), (b,y) \subseteq A$, respectively.  Given two element orbitals $A$ and $B$, we say $A$ {\bf shares an end} with $B$ or $A$ and $B$ {\bf share an end} if $A \subset B$ or $B \subset A$ and at least one of their endpoints is the same.  If $f, g \in PLo(I)$ both have orbital $A$, then we say $f$ and $g$ {\bf share an orbital} or {\bf share the orbital $A$}. If $A$ is an orbital of $f$ and $x \in A$, we say $f$ {\bf moves x to the right} if $xf > x$. We say $f$ {\bf moves x to the left} if $xf < x$.  Continuity implies that if $xf > x$ for some $x \in A$ then $xf > x$ for all $x \in A$.  In this case, we say $f$ {\bf moves points to the right on $A$}, and similarly for left. 

 {\bf Affine components of f} are the components of $[0,1]-B_f$ where $B_f$ is the set of breakpoints of $f$.  Affine components are naturally ordered from left to right.  We may refer to the first or last affine components of $f$ as the {\bf leading affine component} and {\bf trailing affine component}, respectively.  The slopes of the first and last affine components of $f$ are called the {\bf leading and trailing slopes}, also, {\bf initial and terminal slopes} of $f$.  
We often consider a relative version of these definitions on a particular orbital $A$ of $f$, in which case we append {\bf on A} to any of the previous descriptions. Since the elements of our groups are functions which act on the unit interval, we can also discuss whether a relation between functions is satisfied at a point in $I$.  Let $A \subseteq I$. We say that a relation R holds or is true {\bf on $A$} if R is true at every point in $A$. 

The following facts are elementary.

\begin{lem} \mylabel{containment}
If $G$ has no transition chains, and $A,B \in \mathcal{O}(G)$ are such that $A \cap B \neq \emptyset$, then
$A \subset B$, $B \subset A$, or $A=B$.
\end{lem}

\begin{rmk} Let $T$ be a tower in a subgroup $G$ of PLo(I).
 \begin{enumerate}[label={(\arabic*)},ref={\thecor~(\arabic*)}]
 \item Each subset of $T$ is a tower.
 \item $\mathcal{O}(T)$ and each subset of it is a stack.
 \item $\uparrow \mathcal{O}(T)$ in $\mathcal{O}(G)$ is a stack.
 \item $\downarrow \mathcal{O}(T)$ in $\mathcal{O}(G)$ may not be a stack.
 \end{enumerate}
 \end{rmk}

\begin{lem} 
 \begin{enumerate}[label={(\arabic*)},ref={\thecor~(\arabic*)}] 
 \item If $T$ is a tower of elements in PLo(I), then its underlying stack $\mathcal{O}(T)$ is order isomorphic to $T$.
\item  \mylabel{basic conj orbitals} If $g, c \in$ PLo(I) and the orbitals of $g$ are $\{A_i \, | \, 1 \leq i \leq n \}$, then the orbitals of $g^c$ are $\{A_i c \, | \, 1 \leq i \leq n\}$, and the map $\phi: \mathcal{O}(g) \longrightarrow \mathcal{O}(g^c)$ which takes $A_i$ to $A_i c$ is a bijection.
\item \mylabel{basic conj towers} If $T = \{(A_i,a_i) \, | \, i \in I \}$ is a tower in PLo(I) and $c \in$ PLo(I), then the set $T^c \colonequals \{(A_ic,a_i^c) | i \in I \}$ is a tower and the map $\phi_c : T \longrightarrow T^c$ defined by $(O,g) \longmapsto (Oc, g^c)$ is an isomorphism of ordered sets. A similar result holds for stacks.
\item \mylabel{conj slopes} If $g,c \in PLo(I)$, then $g^c$ has the same leading and trailing slopes on each of its orbitals as $g$ has on its corresponding orbitals.  
\item \mylabel{moveover}  Let $G \leq PLo(I)$, $O$ be an orbital of $G$, and $x,y \in O$ with $x < y$.  Then $\exists g \in G$ such that $xg > y$.
\end{enumerate} 
\end{lem}

\begin{lem} \mylabel{multiply basics} Let $G$ be a group without transition chains and $a,b \in G$ with orbitals $O_a, O_b$, respectively.
 \begin{enumerate}[label={(\arabic*)},ref={\thecor~(\arabic*)}]
  \item \mylabel{end}  If $O_a$ and $O_b$ share an end, then $O_a = O_b$.  The contrapositive is also very useful: If $O_a \neq O_b$, then $a$ and $b$ do not share an end.
 \item \mylabel{fd} If $O_a$ is properly contained in $O_b$, then $O_a$ is in a fundamental domain of $(O_b,b)$.  Thus towers in a group without transition chains are always fundamental.
\item \mylabel{multiply} If $O_{a_1} \subset O_{a_2} \subset \cdots \subset O_{a_n}$ is a proper chain of orbitals of $a_1, a_2, \cdots, a_n$, respectively, then $O_{a_n}$ is an orbital of the products $a_1a_2 \cdots a_n$ and $a_na_{n-1}\cdots a_1$.
 \end{enumerate}
 \end{lem}

\mysection{Countability}\mylabel{sectioncountable}

In this section, we prove the following

\begin{thm}\mylabel{countable}
A subgroup of PLo(I) without transition chains is countable.
\end{thm}

 The theorem follows from the next 3 lemmas.  It is elementary to show that a transition chain generates a subgroup of infinite depth, hence is not solvable.  Thus Theorem \ref{countable} easily implies Theorem \ref{countableloc}.  Furthermore, the theorems are actually equivalent due to \ref{loc solv}. \\

Let $\mathcal{T}$ be the collection of all subgroups of PLo(I) without transition chains.  Let $G$ be a group in $\mathcal{T}$.

\begin{lem}\mylabel{towers count}
Towers and stacks in $G$ are countable.
\end{lem}

\begin{proof}
Since every stack is in bijection with some tower by simply picking signatures for each of the orbitals, it is enough to show that the underlying stack of every tower is countable. Let $T$ be a tower in $G$ and $(A,f) \in T$.  Let $A=(a,b)$.  Elements of $T$ are in bijection with elements of $\mathcal{O}(T)$. We produce a collection of disjoint open intervals of $[0,1]$ which are in bijection with elements of $\mathcal{O}(T)$.  Since each of these intervals contains a rational, the set of all the intervals is countable.
Consider the downset of $A$ in $\mathcal{O}(T)$.  We claim there is a $c$ with $a< c < b$ such that $(a,c) \cap C  = \emptyset$ for every $C \in \downarrow A$. The interval $(a,c)$ is the one we seek.  The claimed property of $c$ will imply all these 
intervals are disjoint.

If $f$ moves points left on $A$, then replace $f$ with its inverse.  This does not hinder our argument, since $f$ and $f^{-1}$ have the same orbitals.  Let $(B,g)$ be an element of $\downarrow (A,f)$ in T.  Then $B$ separates $\downarrow A \text{ in } \mathcal{O}(T)$ into two pieces: $\uparrow B$ in $\downarrow A$ and $\downarrow B$ in $\downarrow A$.  Let $x \in B$ and $c = (x)f^{-1}$.  Then $(a,c) \cap B = \emptyset$ due to Lemma \ref{fd}.  Also $(a,c)\cap C = \emptyset$ for any $C \in \downarrow B$. Furthermore, since intervals in $\uparrow B$ contain $x$ and are contained in a fundamental domain of $(A,f)$, they do not contain $(x)f^{-1}$.  Thus intervals in $\uparrow B$ have left endpoints larger than $c$ and so $C \cap (a, c) = \emptyset$ for any $C \in \uparrow B$.  Therefore, $C \cap (a,c) = \emptyset$ for all $C \in \downarrow A$ and the proof is complete. 
\end{proof}

\begin{lem}\mylabel{orbitals count}
The set $\mathcal{O}(G)$ is countable.
\end{lem}
\begin{proof}
Let $L$ be the set of all lengths of elements in $\mathcal{O} (G)$.  $L$ is some subset of $(0,1]$. 
Let
$m : \mathcal{O} (G) \rightarrow L$ be the usual measure on intervals, so $m$ maps each orbital to its length.  

For each positive integer $n$, let $I_n = \left(\left( 2/3 \right)^n, \left(2/3 \right)^{n-1}\right]$.  The set $\displaystyle C = \{I_n | n \in \mathbb{N} \}$ is a partition of the interval $(0,1]$.  Since C is countable, it is enough to show that $m$ maps countably many elements of $\mathcal{O}(G)$ into each element of $C$.  Consider an arbitrary element $I_n$ of $C$, let $R$ be the set of all elements of $\mathcal{O}(G)$ which $m$ maps into $I_n$, and let $K$ be the union of elements of $R$. 

Equip $\mathbb{R}$ with the topology generated by open intervals, and $K$ with the corresponding subspace topology.  Since $\mathbb{R}$ has a countable basis, so does $K$.  Therefore, the open cover $R$ of $K$ has a countable subcover $S$.  If every element of $S$ intersected only countably many elements of $R$, then $R$ would be countable.  Hence, we will show every element of $S$ intersects only countably many elements of $R$.  Let $O$ be an arbitrary element of $S$.
Let $U$ be the subset of $R$ whose elements intersect $O$.   Our aim now is to show $U$ is countable.  We do this by showing $U$ is a stack and therefore countable by the previous lemma.

We must show that every pair of distinct elements $A,B \in U$ are comparable.  Since $G$ has no transition chains, intersection of orbitals implies containment.  Thus we can divide $U - \{O\}$ into two pieces:  $\uparrow O$ and $\downarrow O$, 
or those properly containing $O$ and those properly contained in $O$, respectively. Assume toward a contradiction that $A$ and $B$ are disjoint.  Then they must be contained in $\downarrow O$.  Because they are contained in $O$, one of them will have length less one-half the length of $O$ by Lemma \ref{end}. Assume it is $A$. Since $m(A), m(O) \in I_n$, multiplying their lengths by 2/3 results in a number which is in $I_{n-1}$, hence not in $I_n$.  In particular, $\ds \frac{2}{3} m(O) \leq m(A)$.  Thus $\ds \frac{2}{3} m(O) \leq m(A) \leq \frac{1}{2} m(O)$, a contradiction because $m(A) > 0$.  Therefore $A$ and $B$ are not disjoint, so they must be comparable.
\end{proof}

The next lemma will require the following definition.

Define a {\bf bouncepoint} of a pair $f,g$ of PL functions to be a point $b$ where 1. $(b)f = (b)g$, 2. there is some open interval $(b,c)$ on which $(x)g \neq (x)f$, and 3. $b$ is a breakpoint of $f$ or $g$.  By a {\bf bouncepoint} $b$ of a single function $f$, we mean $b$ is a breakpoint of $f$ and there exists some function $g$ such that $b$ is a bouncepoint of the pair $f,g$.  


 We will also need the following Lemma which is a consequence of results from Section 3.3.2 of  paper \cite{alg}:  
 
\begin{lem}\mylabel{it counts}
Given an orbital $O$ of a group without transition chains, there are at most countably many possible initial and terminal slopes for elements with that orbital. 
\end{lem}

At last, we state the final lemma for our proof of Theorem \ref{countable}.

\begin{lem}\mylabel{signatures count}
Let $F(O)$ be the set of all bumps of functions of $G$ which have orbital $O$.  Then $F(O)$ is countable for any $O \in \mathcal{O}(G)$.
\end{lem}

\begin{proof}
Let $B_n = \{f \in F(O) \, | \, f \text{ has exactly } n \text{ breakpoints}\}$.  Since $\cup^{\infty}_{i=1} B_i = F(O)$, it is enough to show that $B_n$ is countable for $n \in \mathbb{N}$. We define an injective map $\phi$ from the set $B_n$ to a countable set.  Let $f \in B_n$, $x_0$ the left endpoint of $O$, and $s_0$ the initial slope of $f$ leaving $x_0$ (i.e., the initial slope of $f$ on $O$). The function $f$ has $n$ breakpoints, so it has at most $n$ bouncepoints.  Suppose $f$ has $m$ bouncepoints.  Let the bouncepoints of $f$ be $b_1, b_2, . . . ,b_m$, and assume the order on the index set matches that of the points. Let $s_1, s_2, . . . ,s_m$ be the slopes of $f$ leaving $b_1, b_2, . . . , b_m$, respectively, i.e., $s_i$ is the slope of the affine component with left endpoint $b_i$ for $1 \leq i \leq m$. Define $\phi(f)$ to be the ordered set of information $\{s_0, b_1, s_1, b_2, s_2, . . . , b_m, s_m\}$.  

First, we argue $\phi$ is injective.  Assume $f,g \in B_n$ such that $f \neq g $.  If $f,g$ have different initial slopes, we are done, so assume they have the same initial slope.  Consider the maximal closed interval $[x_0, b]$ on which which $f = g$.  Since $f,g$ have the same initial slope, $b$ is a bouncepoint. Note that if $f,g$ don't have the same initial slope, there may not exist any bouncepoint for the pair.  Thus, it is essential that our map $\phi$ include initial slopes. Because $f(x) = g(x)$ for $x<b$, $\phi(f)$ and $\phi(g)$ are the same until the point $b$ appears in one or the other.  At that slot in the ordered sets $\phi(f), \phi(g)$, there are two possibilities: 1. $b$ is a breakpoint (hence a bouncepoint) of exactly one of $f$ or $g$; or 2. $b$ is a breakpoint of both the single functions $f$ and $g$, in which case the next slopes must differ.  In either case, $\phi(f) \neq \phi(g)$.

Now we argue Im($\phi$) is countable by showing that the choices for the $s_i$'s and the $b_i$'s are countable.  Observe that if $b$ is a bouncepoint of some pair $f,g$ in a group $G$ then $b$ is an endpoint of an orbital of $fg^{-1}$.  This is simply because $(b)f = (b)g$ and $(x)f \neq (x)g$ on some interval $(b,c)$.  Thus, by Lemma \ref{orbitals count}, the set $B$ of all possible bouncepoints of elements of $F(O)$ is countable.  Let $S$ be the set of all possible slopes leaving bouncepoints. If there were uncountably many possible slopes for $f$ emanating from some bouncepoint $b$, then applying the chain rule would result in uncountably many initial slopes emanating from $b$ for functions of the form $fg^{-1}$ where $b$ is a bouncepoint of $f,g$ and hence an orbital endpoint for $fg^{-1}$.  This is a contradiction to Lemma \ref{it counts}, so $S$ must also be countable.  Furthermore, all possible initial slopes $S_0$ are countable by Lemma \ref{it counts} (not all orbital endpoints are bouncepoints, so we must note this separately from the previous argument).  Let $S_0, S, B$ include the empty set as an element.  Then it's easy to see that Im$\phi$ injects into the ordered product $S_0 \x B \x S \x B \x S . . . \x B \x S$ where there are $n$ copies of $B \x S$. Therefore, Im($\phi$) is countable.
\end{proof}

\begin{proof}[Proof of Theorem \ref{countable}]
Define $G_n = \{g \in G \, | \, g \text{ has exactly } n \text{ orbitals}\}$.  Then, $\cup^{\infty}_{i=0} G_i = G$, so we need only show $G_n$ is countable.  $G_0$ is just the identity element, so it is countable.  An element of $G_n$ is determined by a choice of $n$ orbitals and a choice of one bump for each of those orbitals. Since each of these $2n$ choices are selected from countable sets by Lemmas \ref{orbitals count} and \ref{signatures count}, $G_n$ injects into a countable set, hence is countable. 
\end{proof} 

To state more consequences, we start with the following definition.

Define a {\bf corner} of a pair $f,g$ of PL functions to be a point $b$ where 1. $(b)f = (b)g$, 2. there is some open interval $(b,c)$ on which $(x)f \neq (x)g$, and 3. $b$ is in the interiors of affine components for both $f$ and $g$.  

\begin{cor}
 The set of all corners of elements of $G$ is countable.
\end{cor}

\begin{proof}
 This follows from Lemma \ref{orbitals count} in the same way that bouncepoints are shown to be countable: If $c$ is a corner of the pair $f,g$ then $g^{-1}f$ has orbital beginning at $c$.  Hence every corner corresponds to some orbital of $G$, the set of which is countable.  
\end{proof}


Interestingly, the proof of \ref{signatures count} demonstrates that for groups with transition chains, corners are not important for distinguishing functions from one another on a single orbital, unlike the case for general subgroups of PLo(I).  However, corners may very well create interesting behavior as functions are multiplied together and create new orbitals.

\begin{cor}
The set of all breakpoints of elements in $G$ is countable. 
\end{cor}

\begin{proof}
Each function in $G$ is completely determined by a finite ordered list $\{b_0, b_1, b_2, \cdots, b_n\}$ of breakpoints.  Since $G$ is countable, the union of all these lists is countable.  Thus the set of breakpoints of elements in $G$ is countable.
\end{proof}

\begin{cor}
Every uncountable subgroup of PLo(I) contains two elements which generate a non-solvable subgroup.
\end{cor}
\begin{proof}
Every uncountable subgroup of PLo(I) contains a transition chain by the contrapositive of Theorem \ref{countable}.  
\end{proof}

\begin{cor}
An ordered wreath product of copies of $\mathbb{Z}$ as defined by P. Hall in \cite{hall} does not embed without transition chains in PLo(I) if the underlying ordered set is uncountable.
\end{cor}

\begin{cor}
If an abstract chain of locally solvable groups has uncountable union, then the chain does not embed in the subgroup lattice of PLo(I).  
\end{cor}

\mysection{Presentations, Representations, and Normal Forms}

We work extensively with iterated wreath products of copies of $\mathbb{Z}$.  This section gives the machinery that we need.

We derive presentations of finitely iterated wreath products with $\mathbb{Z}$.  We apply these to show certain geometric representations of subgroups in PLo(I) are indeed wreath products.  These examples will guide results in section \ref{inducedmaps} and illustrate some of the difficulties in stating such results.  Relations in the presentations will be of central concern in section \ref{non-isomorphism}.  We also give a normal form for elements based on generators of the presentations we derive.

The {\bf restricted wreath product} of $G$ and $H$ denoted $G \wr H$ is the semi-direct product $(\bigoplus_{h \in H} G) \rtimes_{\phi} H$ where $H$ acts on the index set of the sum by right multiplication, permuting copies of $G$.  As we iterate taking wreath products, we do so on the right, collecting parentheses on the left.  For example, for each $n \in \mathbb{N}$, we develop facts about the groups $(\cdots(\mathbb{Z} \wr \mathbb{Z}) \wr \mathbb{Z}) \cdots \wr \mathbb{Z}$ where there are $n$ copies of $\mathbb{Z}$ and $n-1$ wreath products.  We denote this group by $\mathbb{Z} \wr_n \mathbb{Z}$ for brevity as well as clarity and remark $\mathbb{Z} \wr_n \mathbb{Z} \cong W_{n}$ defined in Section \ref{intro}.

\subsection{Presentations}

Since $G \wr H$ is built from a direct sum of copies of $G$ followed by a semi-direct product, we find a presentation for of $\bigoplus_{h \in \mathbb{H}}G$ given a presentation for $G$, and we find a presentation of $G \rtimes_\phi H$ given presentations for $G$ and $H$.  Then we iterate these processes to find a presentation $P_n$, for each $n \in \mathbb{N}$, of $\mathbb{Z} \wr_n \mathbb{Z}$.  The presentation we find for $P_n$ will have finitely many generators.

If $G = <X \, |\, R>$ and $H=<Y \, | \, S>$ are presentations for the groups $G$ and $H$, then a presentation of the semidirect product is
$$G \rtimes_\phi H = <X,Y \, | \, R,S, y^{-1}xy = x(\phi(y)) \, \, \forall \, x \in X \text{ and } y \in Y>.$$ 
A presentation of the countable direct sum is $$\bigoplus_{i \in \mathbb{Z}} G = < \cup_{i \in \mathbb{Z}} X_i \, | \, \cup_{i \in \mathbb{Z}} R_i; [x_i,x_j] = 1 \text{ for all } i \neq j \text{ with } x_i \in X_i \text{ and } x_j \in X_j>$$ where for each $i \in \mathbb{Z}$, $X_i$ is a distinct copy of $X$ and $R_i$ is the corresponding copy of the relations $R$ for the generators in $X_i$.

Using these simple facts, we prove the following lemma:

\begin{lem}
For each $n \in \mathbb{N}$, a presentation of $\mathbb{Z} \wr_n \mathbb{Z}$ is $$P_n = \left<f_1,f_2,\cdots,f_n \,\, | \,\, [f_i,f_j^{w(j+1,n)}] = 1 \text{ for all } i \leq j \text{ with } i,j \in \overline{n-1} \right>$$ where $\ds w(j+1,n) = f_{j+1}^{\alpha_{j+1}} \cdots f_{n-1}^{\alpha_{n-1}}f_n^{\alpha_n}$ with not all $\alpha_j = 0$.
\end{lem}

\begin{proof}
The proof is by induction on $n$.  For $n = 1$, $$P_1 = < f_1 \, | \,>$$ which is a presentation of $\mathbb{Z}$, so the $n=1$ case is true.

Assume the lemma is true for $n$.  We prove it for $n+1$.  The group under consideration is $\mathbb{Z} \wr_{n+1} \mathbb{Z} = (\mathbb{Z} \wr_n \mathbb{Z}) \wr \mathbb{Z} = [\bigoplus_{i \in \mathbb{Z}} (\mathbb{Z} \wr_n \mathbb{Z})] \rtimes \mathbb{Z}$.  Thus we need only use presentations for $\mathbb{Z} \wr_n \mathbb{Z}$ and $\mathbb{Z}$ and the facts discussed before the lemma.

The presentation $$P_n = \left<f_1,f_2,\cdots,f_n \,\, | \,\, [f_i,f_j^{w(j+1,n)}] = 1 \text{ for all } i \leq j \text{ with } i,j \in \overline{n-1} \right>$$ where $\ds w(j+1,n) = f_{j+1}^{\alpha_{j+1}} \cdots f_{n-1}^{\alpha_{n-1}}f_n^{\alpha_n}$ with not all $\alpha_j = 0$ yields  a presentation for the countable direct summand

$$\bigoplus_{k \in \mathbb{Z}} \mathbb{Z} \wr_n \mathbb{Z} = \left< f_{1k},f_{2k},\cdots,f_{nk} : k \in \mathbb{Z} \, \left|  \, \begin{aligned} &[f_{ik},f_{jk}^{w_k(j+1,n)}] = 1 \text{ for all } i \leq j \text{ with } i,j \in \overline{n-1}; \\ &[f_{ik},f_{jl}] = 1 \text{ for } i,j \in \{1,2, \cdots,n\} \text{ and } k \neq l  \end{aligned}  \right. \right>$$ where $\ds w_k(j+1,n) = f_{j+1(k)}^{\alpha_{j+1}} \cdots f_{n-1(k)}^{\alpha_{n-1}}f_{n(k)}^{\alpha_{n}}$ with not all $\alpha_{j} = 0$.

So a presentation of the semidirect product $\mathbb{Z} \wr_{n+1} \mathbb{Z}$ is
$$P = \left<f_{1k},f_{2k}, \cdots, f_{nk},f_{n+1} : k \in \mathbb{Z} \, \left| \, \begin{aligned} &[f_{ik},f_{jk}^{w_k(j+1,n)}] = 1 \text{ for all } i \leq j \text{ with } i,j \in \overline{n-1}; \\ &[f_{ik},f_{jl}] = 1 \text{ for } i,j \in \{1,2,\cdots,n\} \text{ and } k \neq l; \\ &f_{ik}^{f_{n+1}} = f_{i(k+1)} \text{ for } i \in \{1,2, \cdots, n\}\end{aligned} \right. \right>$$ 

where $w_k(j+1,n) = f_{j+1(k)}^{\alpha_{j+1}} \cdots f_{n-1(k)}^{\alpha_{n-1}}f_{n(k)}^{\alpha_{n}}$ with not all $\alpha_{j} = 0$. \\ 

Using the relations $f_{ik} = f_{i0}^{f_{n+1}^k}$ for $k \in \mathbb{Z}$, apply Tietze transformations to get a finitely generated presentation by removing all generators $f_{ik}$ for $k \neq 0$, replacing $f_{ik}$ with $f_{i0}^{f_{n+1}^k}$ in the relations, and removing the relations $f_{ik}^{f_{n+1}} = f_{i(k+1)} \text{ for } i \in \{1,2, \cdots, n\}, k \in \mathbb{Z}$.  Rename $f_{i0}$ as $f_i$ for $i \in \{1,2,\cdots,n\}$.  The result is

$$P' = \left<f_1,f_2, \cdots, f_n,f_{n+1}  \, \left| \, \begin{aligned} &[f_i^{f_{n+1}^k},f_j^{f_{n+1}^kw_k(j+1,n)}] = 1 \text{ for all } i \leq j \text{ with } i,j \in \overline{n-1} \text{ and } k \in \mathbb{Z}; \\ &[f_i^{f_{n+1}^k},f_j^{f_{n+1}^l}] = 1 \text{ for } i,j \in \{1,2,\cdots,n\} \text{ and } k, l \in \mathbb{Z} \text{ such that } k \neq l\end{aligned}\right. \right>$$ 

where $w_k(j+1,n) = (f_{j+1}^{\alpha_{j+1}} \cdots f_{n-1}^{\alpha_{n-1}}f_n^{\alpha_n})^{f_{n+1}^k}$ with not all $\alpha_{j} = 0$.   \\

Note that $f_{n+1}^k w_k(j+1,n) = w(j+1,n)f_{n+1}^k$.  We use this to simplify the first line of relations in $P'$.  We also conjugate the relations in the second line by $f_{n+1}^{-k}$.  The result is:

$$P'' = \left<f_1,f_2, \cdots, f_n,f_{n+1}  \, \left| \, \begin{aligned} &[f_i,f_j^{w(j+1,n)}] = 1 \text{ for all } i \leq j \text{ with } i,j \in \overline{n-1} \text{ and } k \in \mathbb{Z}; \\ &[f_i,f_j^{f_{n+1}^{l-k}}] = 1 \text{ for } i,j \in \{1,2,\cdots,n\} \text{ and } k, l \in \mathbb{Z} \text{ such that } k \neq l\end{aligned}\right. \right>$$ 

Recall the presentation 
$$P_{n+1} = \left<f_1, f_2, \cdots, f_n, f_{n+1} \, | \, [f_i,f_j^{w(j+1,n+1)}] = 1 \text{ for all } i \leq j \text{ with } i,j \in \overline{n} \right>.$$

We wish to show the presentations $P''$ and $P_{n+1}$ are equivalent.  They have the same generators, so we show the sets of relations are equivalent.

Assume the relations of $P_{n+1}$ are true.  When $\alpha_{n+1} = 0$, the word $w(j+1,n+1) = w(j+1,n)$.  Letting $\alpha_{n+1} = 0$ in the relations $[f_i,f_j^{w(j+1,n+1)}]=1$, we get $[f_i, f_j^{w(j+1,n)}] = 1$ where $i,j \in \overline{n-1}$ since $j < m$ where $m$ is the largest subscript of a generator to a non-zero power appearing in $w(j+1,n+1)$.  Taking $\alpha_{j+1}=\alpha_{j+2}= \cdots = \alpha_n = 0$ instead, we get $[f_i,f_j^{f_{n+1}^{\alpha_{n+1}}}] = 1$ where $\alpha_{n+1} \neq 0$ and $i,j \in \overline{n}$. 

Assume the relations in $P''$ are true.  Then since $[f_i, f_j^{w(j+1,n)}]=1$ for all $i \leq j$ with $i,j \in \overline{n-1}$, we have $[f_i,w_j^{w(j+1,n)}]=1$ where $w_j$ is any word in $<f_j,f_{j+1},\cdots f_{n-1}>$.  Also, taking $k=0$, we have $[f_i,f_j^{f_{n+1}^l}]=1$ for any $l \neq 0$.  Therefore, $[f_i,w_j^{f_{n+1}^l}]=1$ where $w_j$ is any word in $<f_j,f_{j+1},\cdots,f_{n}>, l \neq 0$, and $i,j \in \overline{n}$.  Letting $w_j = f_j^{w(j+1,n)}$ where $j \in \overline{n}$, we get $[f_i, (f_j^{w(j+1,n)})^{f_{n+1}^l}] = [f_i, f_j^{w(j+1,n)f_{n+1}^l}] = [f_i, f_j^{w(j+1,n+1)}] = 1$ where we let $\alpha_{n+1} = l \neq 0$.  For the case when $l = 0$, note that $i,j \in \overline{n-1}$ because $j<m$ where $m$ is the largest subscript of a generator to a non-zero power appearing in $w(j+1,n+1) = w(j+1,n)$.  Thus $[f_i,f_j^{w(j+1,n+1)}] = 1$ is still true thanks to the relations $[f_i,f_j^{w(j+1,n)}] = 1$ for all $i \leq j$ with $i,j \in \overline{n-1}$.
\end{proof}

\subsection{Representations} \mylabel{representations}

Now we will use the presentations $P_n$, $n \in \mathbb{Z}$ to prove certain distinct geometric representations of wreath products are isomorphic.  In section \ref{inducedmaps}, we develop facts about maps induced on towers by injective homomorphisms, and these examples will help illustrate limitations of such results as well as existence of certain kinds of towers in a classification.  The names under the diagrams are based on names used in that classification and will assist with referring to the pictures later.

The following lemma and corollaries are used to prove maps defined between representations are homomorphisms.

\begin{lem}\mylabel{wreath}
Let $f$ be a function in PLo(I).
If $H$ is a subgroup of PLo(I) with a single orbital $C$ and if $C$ is contained in a fundamental domain of $f$, then $<H,f> \cong H \wr \mathbb{Z}$.
\end{lem} 

\begin{proof}
For the first part, since $C$ is in a fundamental domain of $f$, we have $Cf^i \cap Cf^j = \emptyset$ for $i \neq j$.  Thus $H$ commutes with $H^{f^i}$ for all $i \in \mathbb{Z}$.  Thus the group $<f>$ acts by conjugation on $H$ to produce the group $B = \bigoplus_{i \in \mathbb{Z}}H$.  This subgroup $B$ is normal in $<H,f>$, $B \cap <f> = \{1\}$, and $<f> \cong \mathbb{Z}$.  Hence $<H,f> \,\,\, \cong \,\,\, B \rtimes <f> \,\,\, \cong \,\,\, H \wr \mathbb{Z}$. 
\end{proof}

Recall, if $T$ is a tower in $G$, then the {\bf type} of the tower T denoted type(T) is the order type of the ordered set of right endpoints of the orbitals in $\mathcal{O}(T)$.  Furthermore, if $T$ is a tower of type $n$, we call refer to $T$ as an {\bf n-tower}.

\begin{cor}\mylabel{stdwr}
A pure fundamental $n$-tower of elements of PLo(I) generates a group isomorphic to $\mathbb{Z} \wr_n \mathbb{Z}$.
\end{cor}
\begin{proof}
This follows from repeated applications of the prior lemma to signatures of the tower.
\end{proof}

\begin{cor}\mylabel{wreathrelns}  Let $(A,f),(B,g),(C,h)$ be pure signed orbitals of PLo(I).
 \begin{enumerate}[label={(\arabic*)},ref={\thecor~(\arabic*)}]
\item If $A$ is contained in a fundamental domain of $(B,g)$, then the relations $[f^{g^i},f^{g^j}]=1 \text{ where } i,j \in \mathbb{Z}$ hold in the group $<f,g>$.
\item If $A$ is in a fundamental domain of $(B,g)$ and $B$ is in a fundamental domain of $(C,h)$, then the relations $[(f^{h^i})^{(g^n)^{h^i}}, (f^{h^j})^{(g^m)^{h^j}}] = [(f^{h^i})^{(g^n)^{h^i}}, g^{h^k}] =[g^{h^i}, g^{h^j}] = 1 \text{ where } i,j,k,n,m \in \mathbb{Z}, i \neq k$ hold in the group $<f,g,h>$.
\end{enumerate}
\end{cor}
\begin{proof}
The group $<f,g,h> \cong <f> \wr <g> \wr <h>$.  Identifying generators of these groups with the appropriate generators in the presentations $P_2$ and $P_3$, we get the desired relations.
\end{proof}

Since many subgroups of PLo(I) have multiple orbitals, it is often useful to project a group $G$ onto one of its orbitals $A$ by taking the action of elements of $G$ to be trivial outside of $A$.  We denote this group by ${\bf G_A}$, and we remark that the map $\phi: G \longrightarrow G_A$ is a homomorphism of groups. Thus if $G$ satisfies some relation $R$, $G_A$ satisfies the corresponding relation $\phi(R)$.  However, it is possible that $G_A$ has more relations than $G$. We can also project $G$ to a union of its orbitals.

\begin{cor}\mylabel{towerrelns} Let $(A,f),(B,g),(C,h)$ be signed orbitals of a group $G \in \mathcal{T}$.
 \begin{enumerate}
\item If $T = \{(A,f),(B,g)\}$ is the a 2-tower and $f$ has no other orbitals in $B$, then the relations $[f^{g^i},f^{g^j}]=1 \text{ where } i,j \in \mathbb{Z}$ are true on the orbital $B$.
\item If $T = \{(A,f),(B,g),(C,h)\}$ is a 3-tower and $f,g$ have no other orbitals in $C$, then the elements $f,g,h$ satisfy the relations $[(f^{h^i})^{(g^n)^{h^i}}, (f^{h^j})^{(g^m)^{h^j}}] = [(f^{h^i})^{(g^n)^{h^i}}, g^{h^k}] =[g^{h^i}, g^{h^j}] = 1 \text{ where } i,j,k,n,m \in \mathbb{Z}, i \neq k$ on the orbital $C$.
 \end{enumerate}
\end{cor}

\begin{proof}
 Follows from the definition of a tower, Lemma \ref{fd}, and applying the previous corollary to the projections $<f,g>_B$ for part 1 and $<f,g,h>_C$ for part 2.
\end{proof}

The previous lemma and corollaries show that in some cases, containment of orbitals of generators is enough to show subgroups are isomorphic to wreath products.  We use this to illustrate examples of some geometric representation of $\mathbb{Z} \wr_3 \mathbb{Z}$. 

For the remainder of this subsection, the groups discussed have no transition chains or, equivalently, are locally solvable.  Recall the collection of all such subgroups in PLo(I) is denoted $\mathcal{T}$.  

By Lemma \ref{wreath}, any group $G$ in $\mathcal{T}$ with poset of orbitals of generators isomorphic to the one in Figure \ref{w3} is isomorphic to $\mathbb{Z} \wr_3 \mathbb{Z}$.  The order of subgroups in the wreath also follows from here; it's $<f_1> \wr <f_2> \wr <f_3>$. 

\begin{figure}[H]\caption{Standard}\mylabel{w3}
\figw
\end{figure}

  Let $G$ be a group in $\mathcal{T}$ with two group orbitals and let $G$ be generated by 3 elements $g_1, g_2$, and $g_3$.  Assume that the generators form a 3-tower $M_1$ on one group orbital and that the order on the generators' subscripts matches the order in the tower. Assume that $g_1$ and $g_3$ form a maximal 2-tower $M_2$ on the other group orbital. Then containment of element orbitals is represented in Figure \ref{split}.

\begin{figure}[H]\caption{Split}\mylabel{split}\centering
\figsplit
\end{figure}

Define the map $\phi: W_3 \longrightarrow G$, $\phi(f_i) = g_i$.  We show that $\phi$ is a homomorphism by showing that for each relation $R$ in $W_3$, the corresponding relation $\phi(R)$ holds in $G$.

By the second part of Corollary \ref{wreathrelns}, $f_1,f_2$, and $f_3$ satisfy the relations $[(f_1^{f_3^i})^{(f_2^n)^{f_3^i}}, (f_1^{f_3^j})^{(f_2^m)^{f_3^j}}] = [(f_1^{f_3^i})^{(f_2^n)^{f_3^i}}, f_2^{f_3^k}] =[f_2^{f_3^i}, f_2^{f_3^j}] = 1 \text{ where } i,j,k,n,m \in \mathbb{Z}, i \neq k$.  We wish to show that when $f$ is replaced with $g$ in these relations, we get the relations that are true in $<g_1,g_2,g_3>$.  By the second part of Corollary \ref{towerrelns}, $g_1,g_2,g_3$ satisfy the relations  $[(g_1^{g_3^i})^{(g_2^n)^{g_3^i}}, (g_1^{g_3^j})^{(g_2^m)^{g_3^j}}] = [(g_1^{g_3^i})^{(g_2^n)^{g_3^i}}, g_2^{g_3^k}] =[g_2^{g_3^i}, g_2^{g_3^j}] = 1 \text{ where } i,j,k,n,m \in \mathbb{Z}, i \neq k$ on the first group orbital.  On the second group orbital, $g_2 = id$.  Using this and the first part of Corollary \ref{wreathrelns}, we conclude all the same relations which hold on the first group orbital also hold on the second group orbital.  Thus $\phi$ is a homomorphism.

The map $\phi$ is injective because $g_1,g_2,g_3$ act geometrically on the first group orbital of $<g_1,g_2,g_3>$ in the same way that $f_1,f_2,f_3$ act on the group orbital of $<f_1,f_2,f_3>$.  It is also surjective, hence an isomorphism.

Now let $G$ be a group generated by elements whose orbitals satisfy any one of Figures \ref{full}, \ref{top}, \ref{free}, or \ref{freecollapse}.  It can be similarly observed that the map $\phi$ corresponding to each of these groups is an isomorphism.

\begin{figure}[H]\caption{Full}\mylabel{full}\centering
   \begin{subfigure}{0.45\textwidth}
     \includegraphics[width=\linewidth]{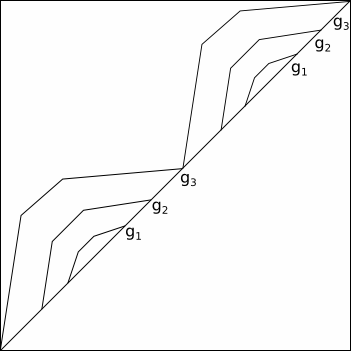}
     \caption{Generators which have a second maximal tower of the same type}\mylabel{plw3_full}
   \end{subfigure}
\hspace{.3cm}
   \begin{subfigure}{0.45\textwidth}
     \includegraphics[width=\linewidth]{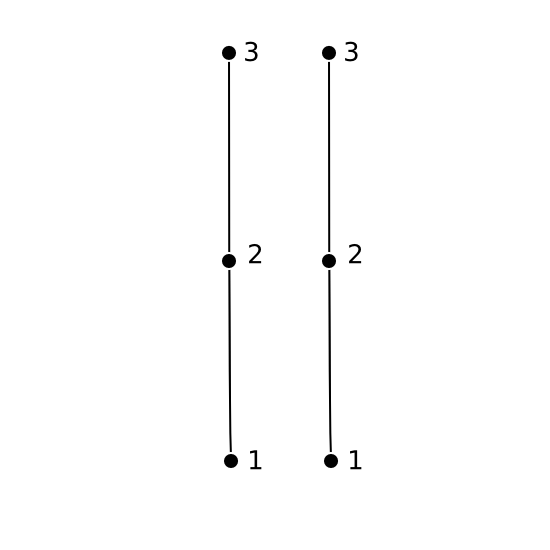}
     \caption{The associated Hasse diagram of orbitals of generators}\mylabel{hw3_full}
   \end{subfigure}
\end{figure}

\begin{figure}[H]\caption{Top}\mylabel{top}\centering
   \begin{subfigure}{0.45\textwidth}
     \includegraphics[width=\linewidth]{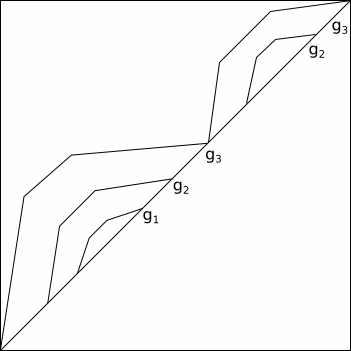}
     \caption{Second maximal tower is at the top when overlayed with the first}\mylabel{plw3_top}
   \end{subfigure}
\hspace{.3cm}
   \begin{subfigure}{0.45\textwidth}
     \includegraphics[width=\linewidth]{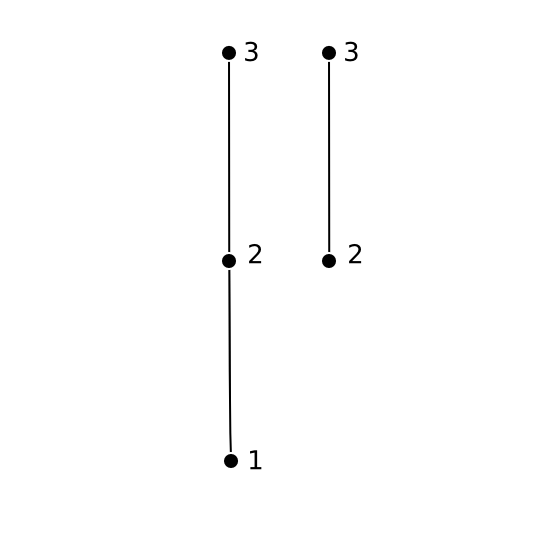}
     \caption{The associated Hasse diagram of orbitals of generators}\mylabel{hw3_top}
   \end{subfigure}
\end{figure}

\begin{figure}[H]\caption{Free}\mylabel{free}\centering
   \begin{subfigure}{0.45\textwidth}
     \includegraphics[width=\linewidth]{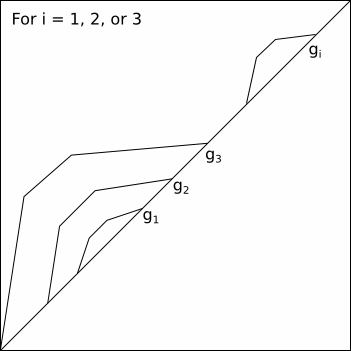}
     \caption{Second maximal tower has only one bump}\mylabel{plw3_free}
   \end{subfigure}
\hspace{.3cm}
   \begin{subfigure}{0.45\textwidth}
     \includegraphics[width=\linewidth]{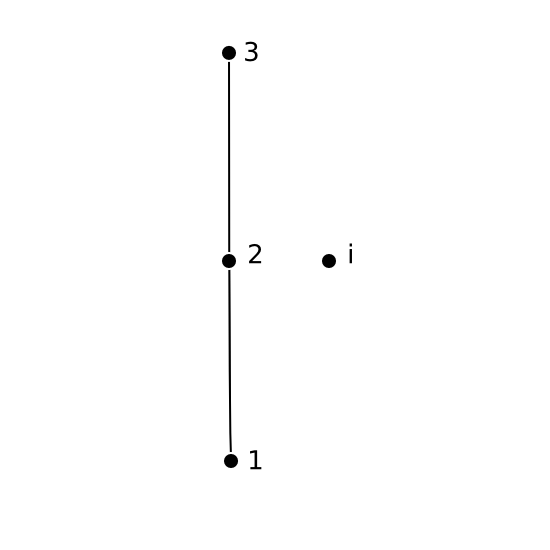}
     \caption{The associated Hasse diagram of orbitals of generators}\mylabel{hw3_free}
   \end{subfigure}
\end{figure}

\begin{figure}[H]\caption{Free Collapse}\mylabel{freecollapse}\centering
   \begin{subfigure}{0.45\textwidth}
     \includegraphics[width=\linewidth]{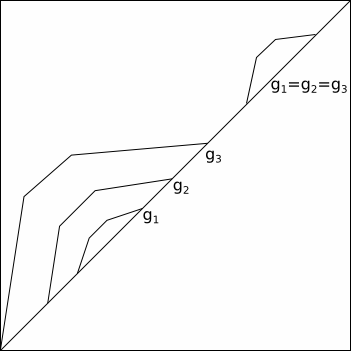}
     \caption{Second maximal tower has one bump, and some elements have been identified}\mylabel{plw3_freecollapse}
   \end{subfigure}
\hspace{.3cm}
   \begin{subfigure}{0.45\textwidth}
     \includegraphics[width=\linewidth]{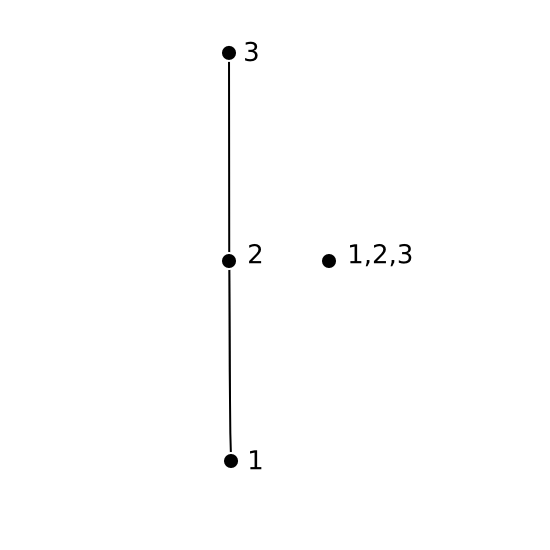}
     \caption{The associated Hasse diagram of orbitals of generators}\mylabel{hw3_freecollapse}
   \end{subfigure}
\end{figure}

These isomorphisms illustrate some of the different geometric representations of $W_3$ that can appear.  They are useful to keep in mind as examples, and we will refer to these and the figures later. \\ \\

Any representation of $W_n$ which is generated by a pure tower, we will refer to as a {\bf standard representation of $W_n$}.

\subsection{Normal Forms}

We now describe a normal form for elements of $\mathbb{Z} \wr_n \mathbb{Z}$ using special subwords we define called syllables.  The normal form is an arrangement of syllables which is determined by relationships between orbitals of generators in a standard representation of $\mathbb{Z} \wr_n \mathbb{Z}$.  To each element, we also associate a unique diagram which represents this arrangement.

Consider a generator $f_i$ in $P_n$. The {\bf algebraic sum} with respect to $f_i$ of $w$ is denoted $\ds \sum_{f_i}w$ and is defined as the sum of all the exponents of appearances of $f_i$ in the word $w$. This notion is well defined on equivalence classes of words if and only if the algebraic sum with respect to $f_i$ of each relator in $P$ is zero.    Thus algebraic sums are well-defined on elements of $P_n$ for each $n \in \mathbb{N}$ since the relations are commutators. 

Note that distinct equivalence classes of words may have the same algebraic sum.  For example, using the presentation $P_2$ $$\ds \mathbb{Z} \wr \mathbb{Z} = <f_1, f_2\, |\, [f_1,f_1^{f_2^{\alpha_2}}]=1 \text{ where } \alpha_2 \in \mathbb{Z}>$$ and a geometric representation of this group generated by a pure 2-tower, we have $f_1 \neq f_2^{-1}f_1f_2$.  However, $\ds \sum_{f_1} f_1 = \sum_{f_1}  f_2^{-1}f_1f_2 = 1$ and $\ds \sum_{f_2} f_1 = \sum_{f_2}  f_2^{-1}f_1f_2 = 0$.  Therefore, given a word $w$, even its spectrum of sums for all elements of the generating set $S$ doesn't distinguish its equivalence class from equivalence classes of other words.  Thus, with respect to the presentation $P_n$, $n \in \mathbb{N}$, we get a well-defined non-injective map $\ds \sum: \,G  \longrightarrow \Pi_{i=1}^n \mathbb{Z}$ where $G = S^* / \sim$ and $n$ is the size of $S$.  In spite of this obstacle, we can distinguish words in $P_n$ from each other.  The process will use algebraic sums in an essential way, but is more complex than considering only algebraic sums.  \\

\begin{lem}\mylabel{normal form}
Each word $w \in P_n$ with $w \neq 1$ has an inductively defined normal form given by
$$w = \ds w_1^{f_n^{k_1}}w_2^{f_n^{k_2}} \cdots w_j^{f_n^{k_j}}f_n^{p_n}$$
where $k_1 < k_2 < \cdots < k_j$, $\ds p_n = \sum_{f_n} w$, each $w_i^{f_n^{k_i}} \neq 1$ and each $w_i$ is in normal form in $P_{n-1}$. If $w = 1$, normal form is the empty word.
\end{lem}

We call $f_n^{p_n}$ the {\bf suffix} of $w$, $p_n$ the {\bf power} of the suffix, each $w_i^{f_n^{k_i}}$ a {\bf prefix}, each $f_n^{k_i}$ a {\bf conjugator}, and each $w_i$ and $w$ itself a {\bf fragment} of $w$. 

As we repeat this process, we consider fragments of fragments, fragments of fragments of fragments, etc, and the same for prefixes, suffixes, etc.  We still refer to each of these simply as prefixes, suffixes, fragments, conjugators, etc., of $w$ and affix a number to them which tracks the point in the inductive process at which they arise.  If $w$ is in normal form, the {\bf layer} of a subword $s$ is the maximum subscript of any generator that appears in $s$.   For example, suffixes and conjugators are powers of generators, so the layer of these is simply the subscript of that generator. The layer of a prefix is the subscript of its conjugator.  If $s$ has layer equal to $l$, we say $s$ is an $l$-fragment or $l$-prefix, etc.     

The {\bf level} of any subword $s$ is max$\{i \, | \, \sum_{f_i} s \neq 0\}$. Note the layer could be distinct from the level.   For example, consider $w=((f_1)^{f_2}f_2^0)^{f_3^0}f_3^0$ in $P_3$.  The level of the fragment $f_1^{f_2}$ is 1, but the layer is $2$.

The {\bf dominant element}  of a subword $s$ of $w$ is $f_k$ where $k =$ level$(s)$.   Syllables of $w$ are special subwords that are either prefixes or fragments with additional properties.  Namely, an $l$-prefix $p$ is a {\bf syllable} if level$(p)$ = $l-1$ and level$(s)=0$ for all suffixes $s$ of fragments containing $p$.  An $l$-fragment $f$ is a {\bf syllable} if level$(f) = l$ and level$(s) = 0$ for all suffixes $s$ of prefixes AND fragments containing $f$.  

For example, if $f_n^{p_n} \neq 1$, $w$ is a {\bf one-syllable} word.  If $f_n^{p_n}=1$, each $w_i^{f_n^{k_i}}$ contains a ${\bf syllable}$ of $w$.

\begin{proof}
We prove this by induction on the presentation subscript.  Suppose $w$ is a nontrivial word in $P_1 = <f_1 \, | \, [f_1,f_1^{w(,)}]=1 >$ where $w(,)$ is the empty word.  Then $w$ is a string of the letters $f_1$ and $f_1^{-1}$.  Its normal form is $f_1^{p_1}$ where $\ds p_1 = \sum_{f_1} w$.  This is unique since $\ds \sum_{f_1}$ is well defined on equivalence classes of words.

Assume inductively that the proposition holds for all $k \in \mathbb{N}$ with $k < n$.  Let $w$ be a nontrivial reduced word in $$P_n = \left<f_1,f_2,\cdots,f_n \,\, | \,\, [f_i,f_j^{w(j+1,n)}] = 1 \text{ for all } i \leq j \text{ with } i,j \in \overline{n-1} \right>$$ where $\ds w(j+1,n) = f_{j+1}^{\alpha_{j+1}} \cdots f_{n-1}^{\alpha_{n-1}}f_n^{\alpha_n}$ with not all $\alpha_j = 0$.  The following process terminates after a finite number of steps.

To put $w$ in normal form, we start with a standard combinatorial group theoretic technique of sliding $f_n$'s to the right.  We first locate all the instances of $f_n$ in the word $w$ and combine adjacent ones into a single power of $f_n$ to get $w = f_n^{a_1}s_1f_n^{a_2}s_2 \cdots f_n^{a_m}s_mf_n^{a_{m+1}}$ where the $s_i$ are reduced words in $\{f_1,f_2, \cdots, f_{n-1}\}^*$ and each $a_i \neq 0$ for all $i$ except possibly $a_1$ and $a_{m+1}$.   Then we move $f_n^{a_1}$ to the right like so:

 $$[f_n^{{a_1}^{-1}},s_1^{-1}]s_1f_n^{a_1}f_n^{a_2}s_2 \cdots f_n^{a_m}s_mf_n^{a_{m+1}}=s_1^{f_1^{-a_1}}f_n^{a_1}f_n^{a_2}s_2 \cdots f_n^{a_m}s_mf_n^{a_{m+1}}$$

 We combine adjacent powers of $f_n$ again.  Now move $f_n^{a_1+a_2}$ to the right.  Repeat the last two steps until there is a single power of $f_n$ all to the way to the right with conjugates of the $s_i$'s by powers of $f_n$ to the left of it, that is, 
$$\ds s_1^{f_n^{-a_1}}s_2^{f_n^{-(a_1+a_2)}} \cdots s_m^{f_n^{-\sum_{i=1}^m a_i}}f_n^{\sum_{i=1}^{m+1} a_i}.$$ 
Cancel each $s_i$ which is trivial for $1 \leq i \leq m$, as well its conjugator and renumber.  We call each $s_i \neq 1$ a {\bf pre-fragment} of $w$.  At this point, the powers of $f_n$ on the pre-fragments may not be in increasing order from left to right.  To put them in nondecreasing order, observe that the pre-fragments with distinct conjugators commute.  If there are both positive and negative powers of $f_n$, we may end up with multiple pre-fragments conjugated by the same power of $f_n$.  Thus, after putting the powers in nondecreasing order, we combine adjacent pre-fragments which are conjugated by the same power of $f_n$ into a single subword conjugated by that same power of $f_n$.  We obtain an expression like:

$$w = \ds w_1^{f_n^{k_1}}w_2^{f_n^{k_2}} \cdots w_l^{f_n^{k_l}}f_n^{p_n}$$

where $k_1 < k_2 < \cdots < k_l$, $\ds p_n = \sum_{f_n} w = \sum_{i=1}^{m+1} a_i$ and each $w_i$ is in $\{f_1,f_2,\cdots,f_{n-1}\}^*$.  

Cancel each $w_i$ which is trivial for $1 \leq i \leq l$, as well its conjugate and renumber.   By induction, each $w_i$ has a normal form in $P_{n-1}$.  Putting each $w_i$ in normal form, we complete the process of putting $w$ in normal form and obtain

$$w = \ds w_1^{f_n^{k_1}}w_2^{f_n^{k_2}} \cdots w_j^{f_n^{k_j}}f_n^{p_n}$$

where $k_1 < k_2 < \cdots < k_j$, $\ds p_n = \sum_{f_n} w$, each $w_i^{f_n^{k_i}} \neq 1$ and each $w_i$ is in normal form in $P_{n-1}$.

Each remaining $w_i^{f_n^{k_i}}$ is a {\bf prefix} of $w$. 

To prove uniqueness, observe that the form of each $w_i$ for $1\leq i \leq j$ is unique by induction.  The word $\ds w_1^{f_n^{k_1}}w_2^{f_n^{k_2}} \cdots w_j^{f_n^{k_j}}$ lies in $\ds \bigoplus_{i=1}^{\infty} \mathbb{Z} \wr_{n-1} \mathbb{Z}$.  It is in normal form in the presentation of $\ds \bigoplus_{k \in \mathbb{Z}} \mathbb{Z} \wr_{n-1} \mathbb{Z}$ below:
 $$\left< f_1^{f_{n}^k},f_2^{f_{n}^k},\cdots,f_{n-1}^{f_{n}^k} : k \in \mathbb{Z} \, \left| \, \begin{aligned} &[f_i^{f_{n-1}^k},f_j^{f_{n-1}^k w_k(j+1,n-1)}] = 1 \text{ for all } i \leq j \text{ with } i,j \in \overline{n-1}; \\ &[f_i^{f_{n-1}^k},f_j^{f_{n-1}^l}] = 1 \text{ for } i,j \in \{1,2, \cdots,n\} \text{ and } k \neq l  \end{aligned}  \right. \right>$$ 
 where $\ds w_k(j+1,n-1) = (f_{j+1}^{\alpha_{j+1}} \cdots f_{n-1}^{\alpha_{n-1}}f_{n-1}^{\alpha_{n-1}})^{f_n^k}$ with not all $\alpha_{j} = 0$
thanks to the increasing exponents $k_1 < k_2 < \cdots < k_j$ of $f_n$ and since each $w_i \neq 1$.  Hence the subword $w_1^{f_n^{k_1}}w_2^{f_n^{k_2}} \cdots w_j^{f_n^{k_j}}$ is unique.  In a semi-direct product $A \rtimes B$ each word can be written in the form $ab$ where $a,b$ is in the generating set for $A,B$, respectively.  The form $w = \ds w_1^{f_n^{k_1}}w_2^{f_n^{k_2}} \cdots w_j^{f_n^{k_j}}f_n^{p_n}$ is one such example in the semi-direct product $<f_1, f_2, \cdots, f_{n-1}> \wr <f_n>$.  Since the form of the subwords  $w_1^{f_n^{k_1}}w_2^{f_n^{k_2}} \cdots w_j^{f_n^{k_j}}$ and $f_n^{p_n}$ are unique and since these subwords do not commute unless one is trivial, the form of their product is unique.
\end{proof}

Though unnecessary for the previous proof, it is insightful to consider in more detail how further steps of the inductive process work and establish more terminology and notation.   For example, to put each $w_i$ in normal form, we apply the same process but with $w_i$ playing the role of $w$ and $f_{n-1}$ playing the role of $f_n$. We obtain

$$w_i = \ds w_{i,1}^{f_{n-1}^{k_{i,1}}}w_{i,2}^{f_{n-1}^{k_{i,2}}} \cdots w_{i,j_i}^{f_{n-1}^{k_{i,j_i}}}f_{n-1}^{p_{i,n-1}}$$

 where $w_{i,j} \in \{f_1,f_2,\cdots,f_{n-2}\}^*$, $k_{i,1}<k_{i,2}< \cdots < k_{i,j_i}$,\\
 and $\ds p_{i,n-1} = \sum_{f_{n-1}}w_i$.  
 
Each $w_{i, l}$ and $w_i$ itself is called a {\bf fragment} of $w_i$ or an $(n-2)$-{\bf fragment}, $(n-1)$-{\bf fragment} of $w$ to emphasize the layer in $w$. We refer to each $w_{i,l}^{f_{n-1}^{k_{i,l}}}$ as a {\bf prefix} and to $f_{n-1}^{p_{i,n-1}}$ as the {\bf suffix} of $w_i$.  They are $(n-1)$-prefixes and $(n-1)$-suffixes of $w$, respectively.  If the suffix $f_{n-1}^{p_{i,n-1}} \neq 1$, $w_i$ is a {\bf one-syllable } word.   Otherwise, each prefix of $w_i$ contains a {\bf syllable} of $w_i$.  

We repeat the inductive process on the prefixes of $w_i$, each time appending a new number to the ordered list in the subscript to track what prefix we are in at the current layer.  Therefore $n+1-k$ indicates the layer of the prefixes and suffixes arising in step $k$ of the inductive process.  By our notation, the number $k$ also equals the number of numbers in the subscript of the fragments and powers of the conjugators and the suffix.

Now that we have shown existence of normal forms in $P_n$, we define the {\bf level} of a group element $g \in G$ where $G$ has presentation $P_n$ to be the level of its normal form.  Similarly, we use terminology such as {\bf prefixes}, {\bf suffixes}, {\bf fragments},and  {\bf syllables} of $g$ to refer prefixes, suffixes, fragments, and syllables of the normal form of $g$.  Later we will write level$_G(g)$ to indicate the group is $G$ when working with levels in multiple groups.

\subsection{Diagrams Corresponding to Normal Forms}
Let $M_n$ be the rooted $\mathbb{Z}$-ary tree of height $n$ with a distinguished 0-edge below each vertex and adjacent vertices labeled in increasing order from left to right with the elements of the integers. We call $M_n$ the {\bf mother tree}.  We describe diagrams which correspond to normal forms of words in $P_n$ and are subdiagrams of $M_n$ with additional labels. Let $F$ be the collection of all finite sub rooted tree diagrams of $M_n$, and let $W$ be all possible trees of $F$ resulting from labeling leaves of elements of $F$ with elements of $\mathbb{Z} - \{0\}$ and nonterminal vertices by elements of $\mathbb{Z}$.

\begin{lem}\mylabel{tree diagram}
There is a one-to-one correspondence between elements of $W$ and normal forms of words in $P_n$.
\end{lem}

\begin{proof}
Given a word $w = \ds w_1^{f_n^{k_1}}w_2^{f_n^{k_2}} \cdots w_j^{f_n^{k_j}}f_n^{p_n}$ in normal form in $P_n$, we obtain a tree in $W$ in the following way: The vertices correspond to suffixes of $w$ and are labeled with the power of the corresponding suffix.  The edges of the graph correspond to conjugators and are labeled with the power of the corresponding conjugator.  Let $s_A$ be the suffix corresponding to a vertex $A$.  There is a directed edge from $A$ to $B$ if layer($s_A$) = layer($s_B$) + 1. Now let depth($A$) = the length of the path from the root of the tree to $A$ and height($A$) = n - depth($A$). Then layer($s_A$) = height($A$). Since a rooted tree is naturally graded by the height function on the vertices, it is graded by the layer function on the labels of the vertices.  Similarly, we define depth($E$) for an edge $E$ to be the depth of its terminal vertex and height($E$) = n - depth($E$). To make our diagrams easier to read, we omit labels of zero and we also circle edge labels. 

For example, given the word 
$$g =\ds \bigg[\left((f_2^{f_3^0}f_3^0)^{f_4^{-2}}(f_2^{f_3^0}f_3^0)^{f_4^2}\right)^{f_5^0}f_5^0\bigg]^{f_6^0}\bigg[\left((f_2^{f_3^0}f_3^0)^{f_4^0}(f_2^{f_3^0}f_3^0)^{f_4}f_4\right)^{f_5^0}f_5^0\bigg]^{f_6^2}f_6$$ which is in normal form in $P_6$, the corresponding tree diagram is Figure \ref{tree_ex1}.

\begin{figure}\centering
     \includegraphics[width=.5\linewidth]{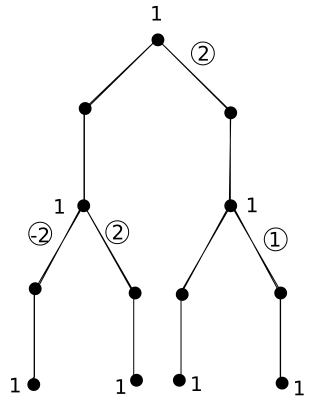}
     \caption{Tree diagram for the normal form of $g$}\mylabel{tree_ex1}
\end{figure}

 We say that a vertex $v$ is {\bf above} another vertex $u$ if height($v$) $>$ height($u$).  The vertex $v$ is {\bf above} an edge $E$ if height($v$) $\geq$ height($E$).  Similarly, $E$ is above another edge $F$ if height($E$) $>$ height($F$).  The edge $E$ is above a vertex $v$ if height($E$) $>$ height($v$).  We say an edge or vertex is {\bf below} another if the opposite inequalities hold for the height.  Note that these inequalities also imply corresponding inequalities for depth.  Furthermore, when we refer to {\bf distance} between vertices or between edges, we are invoking the usual metric on trees (length of the unique path between the vertices or between the edges).

There is also an obvious {\bf word diagram} associated to the word $w$ which is isomorphic as a labeled graph to the tree defined previously. Instead of labeling with the powers of suffixes and conjugators, we instead label with the entire suffixes and conjugators. Since a tree diagram can be obtained from its associated word diagram by forgetting only information which is captured by the depth function of the tree (namely subscripts of conjugators and suffixes), there is a unique tree diagram associated to each word diagram and vice versa.  Furthermore, given a tree diagram in $W$, one can construct its associated word diagram by using the height of vertices and edges to affix subscripts to generators, and the normal form is easily obtained from the word diagram.  

More specifically, given  a tree $T$ and its associated word diagram $W$, the word $w$ in normal form corresponding to $W$ is obtained by this process:  The label of the root is the suffix of $w$.  An edge attached to the root and the maximal subtree below it correspond to a prefix, so we call this a {\bf prefix subdiagram} or {\bf n-prefix subdiagram} to emphasize the layer.   In general, the diagram hanging below a vertex of height $k$ (not including the vertex itself)  is called a {\bf k-prefix subdiagram}.  Similarly, the subtree hanging at a vertex of height $k$ (including the vertex itself) is called a {\bf fragment subdiagram} or {\bf k-fragment subdiagram}.

After writing the suffix of $w$, start with the rightmost of the prefix subdiagrams and write parentheses to the left of the suffix with blank space inside (for writing down a fragment).  Conjugate the parentheses by the label of the edge adjacent to the root.  Do this for each prefix diagram, writing right to left so that the prefix subdiagrams are in the same order as the prefixes (necessary for the word to be in normal form).  At this point, the blanks in the parentheses each correspond to rooted subtrees of the word diagram.  Repeat the same process for each of these blanks. Since the layer decreases with each step, this process terminates.  Furthermore the word is in the proper form up to fragments, and repeating the process results in fragments which are in normal form.  Thus the word $w$ is in normal form.
 
 Given two distinct tree diagrams $T_1$ and $T_2$, the words $w_1$ and $w_2$ obtained from the trees using the prior process are in normal form.  Hence, they are unique.  Conversely, given two distinct words $w_1$ and $w_2$ in normal form their associated tree diagrams $T_1$ and $T_2$ are unique or else the previous process would result in the same normal form for $w_1$ and $w_2$.
\end{proof}

\subsection{Geometric Interpretations of Normal Forms}

In this section, we connect normal forms and tree diagrams with the standard geometric representation of $W_n$.  Throughout, we assume that $G$ is a group generated by a pure fundamental tower $T = \{(A_1, f_1) \cdots (A_n, f_n)\}$, i.e., we are considering a standard representation of $W_n$ in PLo(I). 

\begin{lem}
\begin{enumerate}[label={(\arabic*)},ref={\thecor~(\arabic*)}]
\item\mylabel{fragbump} A product of the form $$w_1^{f_i^{k_1}}w_2^{f_i^{k_2}} \cdots w_j^{f_i^{k_j}}f_i^{k}$$ with $k\neq 0$ is a one-bump function with orbital $A_k$. 
\item\mylabel{prefixbump} A product of the form $$(w_1^{f_i^{k_1}}w_2^{f_i^{k_2}} \cdots w_j^{f_i^{k_j}}f_i^{k})^{w(i+1, n)}$$ where $k \neq 0$, $\alpha_i \in \mathbb{Z}$, and $w(i+1, n) = f_{i+1}^{\alpha_{i+1}} \cdots f_n^{\alpha_n}$,  is a one-bump function with orbital $A_kw(i+1, n)$.
\end{enumerate}
\end{lem}
\begin{proof}
\ref{fragbump} This is a simple application of \ref{multiply} to a fragment of a word in $P_n$ in a standard representation of $W_n$.

\ref{prefixbump} Apply \ref{multiply} and \ref{basic conj orbitals}.


\end{proof}

\begin{lem} Given a word $w = \ds w_1^{f_n^{k_1}}w_2^{f_n^{k_2}} \cdots w_j^{f_n^{k_j}}f_n^{p_n}$ in normal form, there is a one-to-one correspondence between bumps of $w$ and maximal non-zero vertices in the tree diagram of $w$. 
\end{lem}
\begin{proof}

Using the tree $T$ associated to $w$, we find the bumps of $w$ in the following way.  Consider all paths from the root of the tree to leaves.  Let $p$ be one such path, and let $v$ be the first vertex along $p$ with a non-zero label.  The fragment subdiagram (possibly the whole tree $T$) hanging at $v$ corresponds to a bump of $w$.  To see this note that the vertices on $p$ above $v$ represent the trivial element.  The tree hanging at $v$ represents a fragment $f$ of $w$ with non-trivial suffix.  Furthermore, all suffixes corresponding to vertices above $v$ on $p$ are trivial.  Let the nontrivial conjugators corresponding to edges above $v$ on $p$ be $f_i^{k_i}, f_{i+1}^{k_{i+1}}, \cdots , f_n^{k_n}$.  Then $\ds f^{f_i^{k_i} f_{i+1}^{k_{i+1}} \cdots f_n^{k_n}}$ is a word representing a one-bump function by \ref{prefixbump}. Furthermore, $\ds f^{f_i^{k_i} f_{i+1}^{k_{i+1}} \cdots f_n^{k_n}}$ is a bump of $w$ itself because paths that diverge from $p$ above the vertex $v$ correspond to subwords of $w$ conjugated by a distinct power of at least one of $f_i, f_{i+1}, \cdots, f_n$ and thus have support disjoint from $\ds f^{f_i^{k_i} f_{i+1}^{k_{i+1}} \cdots f_n^{k_n}}$.

Thus the fragment hanging at $v$ corresponds to a bump of $w$. This process always results in a bump of $w$ because each path from the root to a leaf of $T$ cannot have all vertices labeled zero.  To see this, let $p$ be such a path.  If the prefix subdiagram hanging from the root had all 0-labeled vertices, then the prefix would be trivial and thus $w$ would not be in normal form (a contradiction).

Conversely, given a bump $b$ of $w$, there is some fragment $$f = w_1^{f_i^{k_1}}w_2^{f_i^{k_2}} \cdots w_j^{f_i^{k_j}}f_i^{k}$$ of $w$ whose graph contains the graph of $b$ in the plane.  Let $f$ be the fragment of minimal layer with this property, $v$ be the vertex in the tree diagram of $w$ corresponding to the suffix $f_i^{k}$, and $p$ the path from $v$ to the root. Note that the fragment $f$ is unique since distinct fragments in the same layer have disjoint supports.  If $f_j^l$ is a suffix corresponding to a non-zero vertex $u$ above $v$, by \ref{multiply} the product $\ds f' $ corresponding to the fragment subdiagram hanging at $u$ and thus the entire word $w$ would have orbital containing the orbital of $b$, a contradiction.  Therefore all vertices above $v$ on the path $p$ from $v$ to the root are labeled with zero. If $k \neq 0$, $f$ is a one-bump function by \ref{fragbump} and $v$ is the maximal non-zero vertex we associate to $b$.  If $k =0$, then since distinct $i$-prefixes have disjoint supports, there exists some $i$-prefix $r$ whose graph contains the graph of $b$.  If the $(i-1)$-fragment $s$ of $r$ had trivial suffix $t$, then $f$ would not be the fragment of minimal layer whose graph contains the graph of $b$. Thus $s$ does not contain the graph of $b$, so $t \neq 1$.  Thus the vertex $u'$ corresponding to $t$ has non-zero label, and we associate this vertex to $b$ when $k = 0$.
\end{proof}

In any tree diagram, we call a maximal non-zero vertex $v$ a {\bf bump vertex} and the fragment subdiagram hanging at $v$ a {\bf bump diagram}.

In the following corollary, we say stacks $S_1$ and $S_2$ are {\bf conjugate} in $G$  to mean there are towers $T_1$ and $T_2$ associated to $S_1$ and $S_2$ such that $\mathcal{O}(T_1^c) = \mathcal{O}(T_2)$ for some $c \in G$.  Equivalently, there is a $c \in G$ such that $S_1c=S_2$.

\begin{cor}\mylabel{levelorb}
\begin{enumerate}[label={(\arabic*)},ref={\thecor~(\arabic*)}]
\item \mylabel{allbumps}If $b$ is a single bump of an element $g \in G$, normal form of $b$ can be simplified to $$b = \ds (w_1^{f_i^{k_1}}w_2^{f_i^{k_2}} \cdots w_j^{f_i^{k_j}}f_i^{k})^{w(i+1, n)}$$
where $1 \leq i \leq n$, $k_1 < k_2 < \cdots < k_j$, $\ds k = \sum_{f_i} w$, $k \neq 0$, each $w_i^{f_n^{k_i}} \neq 1$, each $w_i$ is in normal form in $P_{i-1}$, and $w(i+1, n) = f_{i+1}^{\alpha_{i+1}} \cdots f_n^{\alpha_n}$.  In particular, each bump $b$ is also an element of $G$.
\item Given a word $w = \ds w_1^{f_n^{k_1}}w_2^{f_n^{k_2}} \cdots w_j^{f_n^{k_j}}f_n^{p_n}$ in normal form in $P_n$ geometrically represented as a graph of a function in a standard representation of $W_n$, there is a one-to-one correspondence between syllables of $w$ and bump vertices of $w$. 
\item \mylabel{mother} The unlabeled mother tree $M_n$ is the Hasse diagram for the poset $\mathcal{O}(G)$.
\item  \mylabel{max iso} All maximal stacks are conjugate in $G$ and therefore order isomorphic. Since maximal stacks are isomorphic, maximal towers are, too.
\end{enumerate}
\end{cor}

 Let $O$ be an element orbital of $g$ in the solvable group $G$.  The signed orbital $(O,g)$ is contained in some maximal tower $M$ of $G$.  Let $\phi$ be the ordered set isomorphism from $M$ to $\overline{n}$. Define the {\bf level} of the element orbital $O$ in $G$, denoted level$(O)$ to be $\phi(O, g)$.  This notion is well defined in standard representations of $W_n$ via Lemma \ref{max iso}.  It can also be defined in solvable groups in general thanks to results in \cite{alg}.  Namely, each solvable subgroup embeds in its split group, the split group has the same element orbitals, apply the previous lemma to the projection of the split group to its group orbital containing $O$, then use that the split group is represented by a disjoint union of standard representations.

\begin{lem}
Given two words $w_1$ and $w_2$ in normal form in $P_n$, their associated trees $T_1$ and $T_2$ can be used to determine containments of orbitals associated to the words.
\end{lem}

\begin{proof}
Suppose $(B_1, b_1)$ and $(B_2, b_2)$ are signed orbitals of $w_1$ and $w_2$, respectively, such that $b_1 \neq b_2$.  Consider the bump vertices $v_1$ and $v_2$ of $b_1$ and $b_2$ and let $p_1$ and $p_2$ be the unique paths from $v_1, v_2$ to the roots of the trees $T_1, T_2$.  

Embed $T_1$ in the mother tree:  Identify the roots, then identify edges at height $n-1$ if they have the same label, and identify terminal vertices of the edges which were identified.  Repeat this process with the trees hanging on each of the new vertices at height $n-1$.  Since the height decreases with each iteration, this process terminates with an embedding of $T_1$ into $M_n$.  Do the same with $T_2$.  Consider the paths $p_1$ and $p_2$  as embedded in $M$.  

Let $v$ be the vertex where $p_1$ and $p_2$ diverge from one another.  The vertex $v$ exists since $p_1 \neq p_2$ and paths between vertices in a tree are unique.  Let $l_1$ and $l_2$ be the labels of the vertices corresponding to $v$ in $T_1$ and $T_2$, respectively.  If exactly one of $l_i = 0$, say $l_1$, then $B_1 \subset B_2$.  If both are zero, then $B_1 \cap B_2 = \emptyset$.  If neither are zero, then $B_1 = B_2$.
\end{proof}

 The following lemma connects the geometric and algebraic definitions of level in a standard representation.

 \begin{lem}\mylabel{level}
 Let $G$ be a solvable subgroup of PLo(I) generated by the signatures of a pure fundamental tower $T = \{(A_i, f_i) \,| \, 1 \leq i \leq n\}$ in order.
\begin{enumerate}[label={(\arabic*)},ref={\thecor~(\arabic*)}]
 \item\mylabel{level1} level$(A_i)$ = level$(f_i) = i $
 \item\mylabel{level2} level$(A_ic) = $level$(f_i^c) = $level$(f_i) = i$ for all $c$ in $G$. 
 \item\mylabel{level3} The level of a bump $b$ of an element $g \in G$ equals the level of its orbital $B$. 
 \item\mylabel{level4} The level of any $g \in G$ is max$\{\text{level}(O) \, | \, \text{$O$ is an orbital of $g$}\}$.
 \end{enumerate}
 \end{lem}
 
 \begin{proof}
\ref{level1} The subscript $i$ is by definition the level of $f_i$.  By choice of notation, it is also the slot of the signed orbital $(A_i, f_i)$ in the ordered set $T$.  \\

\ref{level2} Since $\hat{c}: T \longrightarrow T^c$ is an order isomorphism by \ref{basic conj towers}, the second point is also true.  \\

\ref{level3} Since $b$ is a single bump, it is represented by a one-syllable word $w$ in normal form with respect to $f_i, 1 \leq i \leq n$.  Suppose in the process of putting $b$ into normal form, the first non-trivial suffix of $b$ is $f_i^k$ for some $k \in \mathbb{Z}$ and some $1 \leq i \leq n$, so the dominant element of $b$ is $f_i$.  Since any conjugators with higher subscript contribute an algebraic sum of 0 , level$(b) \leq i$.  Since $f_i$'s left of the suffix $f_i^k$ arise from conjugation, they make a net contribution of $0$ to $\sum_{f_i} b$, so this sum equals $k$.  Thus level$(b) = i$. We wish to show level$(B) = i$.  
 
Since $b$ is a single bump with the first non-trivial suffix equal to $f_i^k$, normal form of $b$ can be simplified to $$b = \ds (w_1^{f_i^{k_1}}w_2^{f_i^{k_2}} \cdots w_j^{f_i^{k_j}}f_i^{k})^{w(i+1, n)}$$

where $k_1 < k_2 < \cdots < k_j$, $\ds k = \sum_{f_i} w$, each $w_i^{f_n^{k_i}} \neq 1$, each $w_i$ is in normal form in $P_{i-1}$, and $w(i+1, n) = f_{i+1}^{\alpha_{i+1}} \cdots f_n^{\alpha_n}$. 

Then $\ds b^{w(i+1, n)^{-1}} = w_1^{f_i^{k_1}}w_2^{f_i^{k_2}} \cdots w_j^{f_i^{k_j}}f_i^{k}$.  Note that orbitals of $i$ fragments are properly contained in the orbital of the corresponding $i$ suffix.  Note also that $f_i^k$ has orbital $A_i$. Thus by \ref{multiply}  $w_1^{f_i^{k_1}}w_2^{f_i^{k_2}} \cdots w_j^{f_i^{k_j}}f_i^{k}$ has orbital $A_i$.  The orbital $(A_i, f_i)$ is in $T$. Alter its signature to be $w_1^{f_i^{k_1}}w_2^{f_i^{k_2}} \cdots w_j^{f_i^{k_j}}f_i^{k}$, and call the new tower $T'$.   By \ref{basic conj towers}, the maximal induced map $\ds \hat{w}(i+1, n): T' \longrightarrow T'^{w(i+1, n)}$ is an isomorphism of ordered sets.  By the previous analysis, it maps $(A_i, w_1^{f_i^{k_1}}w_2^{f_i^{k_2}} \cdots w_j^{f_i^{k_j}}f_i^{k})$ to $(B, b)$.  Thus level$(A_i) = $level$(B)$, so level$(B) = i$. \\

\ref{level4} By \ref{level3}, each bump and hence each syllable of $g$ has level equal to the level of its orbital.  By definition, level($g$) = max$\{i \, | \, \sum_{f_i} g \neq 0\}$.  Since the syllables of $g$ partition the normal form of $g$, level($g$) = max$\{i \, | \, level(s) = i \text{ and $s$ is a syllable of $g$} \}$ = max$\{i \, | \, \text{level}(O) = i \text{ and $O$ is the orbital of a syllable of $g$}\}$  = max$\{i \, | \, \text{level}(O) = i \text{ and $O$ is an orbital of $g$}\}$.
 \end{proof}
 
 \begin{cor}\mylabel{conj level}  Let $G$ be a solvable subgroup of PLo(I) generated by the signatures of a pure fundamental tower $T = \{(A_i, f_i) \,| \, 1 \leq i \leq n\}$ in order.
 \begin{enumerate}[label={(\arabic*)},ref={\thecor~(\arabic*)}]
 \item The subscript of the dominant element of a bump is the level of both the bump itself and its orbital.
 \item If level($b$) = $i$ where $b$ is a bump of $g \in G$, there exists an element $c \in G$ of the form $w(i+1,n)=f_{i+1}^{k_{i+1}}f_{i+2}^{k_{i+2}} \cdots f_n^{k_n}$ which conjugates the orbital $B$ of $b$ to the orbital $A_i$ of a generator $f_i$ and such that $b^c$ has dominant element $f_i$.
 \end{enumerate}
 \end{cor}

\mysection{Constructing Locally Solvable Groups}\mylabel{sectionconstruction}

Recall we defined $W_n$ in Section \ref{intro}.  We now use it to state our construction result.

\begin{thm}\mylabel{construction}
For each countable ordered set $C$, there is a locally solvable subgroup $W_C$ in Thompson's Group F with a pure tower of generators order isomorphic to $C$.  Furthermore, any set of $n$ signatures of the generating tower generates a copy of $W_n$ where $n \in \mathbb{N}$.
\end{thm}

\begin{proof}
Build a Cantor set in $[0,1]$, but instead of removing the middle third of each interval at each stage, remove the middle half of each interval at each stage.  Note that each removed interval has dyadic endpoints and its midpoint is also dyadic.  Let $A$ be the intervals deleted confined to $[3/4, 1]$.  Let $M$ be the set of midpoints of the deleted intervals confined to $[3/4, 1]$.  $M$ is a subset of the dyadic rationals $D$.  Furthermore, if $x \in M$, there is a deleted open interval $I_x$ in $A$ with endpoints in $D$ which contains $x$ but no other points of $M$.  Also, $M$ is a dense countable ordered set, meaning if $x,y \in M$, there exists a $z \in M$ such that $x < z < y$.  

For $x \in [3/4,1]$, let $x'=1-x$, the ``reflection" of x across 1/2.  For a subset A of $[3/4,1]$, let $A'=\{x' | x \in
 A\}$.  We will consistently use letters without primes for elements of $(1/2,1]$ and letters with primes for elements of $[0, 1/2)$.  Similarly, a subset $A'$ contained in $[0, 1/2)$ will have a corresponding $A=A''$ in $(1/2, 1]$.

We have the following properties of $M$:  1. $M$ is a dense ordered set, 2. $M$ is a subset of the dyadic rationals, and 3. Each point $x \in M$ has a deleted interval $I_x$ around it with dyadic endpoints and which contains no other points of $M$. Let $r_x$ be the right endpoint of $I_x$.  Then $r_x'$ is the left endpoint of $I_x'$ for each $x \in M$.  

Since $M$ is a countable dense total order, there is a subset $B$ of $M$ which is order isomorphic to $C$.  Given $x \in B$ construct an element $f_x$ of Thompson's Group $F$ such that $f_x(x) = x, f_x(r_x')=r_x, f_x(x')=x'$, and $f_x(a) = a$ for points $a \in [0,x']\cup[x,1]$.  Such an element exists because $F$ acts $n$-transitively on the set of dyadic rationals in $[0,1]$ (see Lemma 4.2 of \cite{cfp}). It is elementary to require that $tf_x \ne t$ for all t in (x',x).  Let $W_C$ be the generated by the set of all $f_x$ for each $x \in B$.  We claim $W_C$ is locally solvable.  

Given a function $f_x \in W_C$, its graph contains the points $(x,x), (r_x', r_x), (x', x')$.  Since $f_x$ is an orientation-preserving bijection, the graph of $f_x$ is contained between the triangle $T_x$  with vertex set $\{(x, x), (x', x), (x', x')\}$ and $T_x'$ with vertex set $\{(r_x, r_x), (r_x', r_x), (r_x', r_x')\}$.  Furthermore, the graph of $f_x$ only intersects the triangles in the vertices $(x,x), (r_x', r_x), (x', x')$.  By definition, the function $f_x$ is non-trivial on the interval $(x', x)$ and trivial elsewhere, so it is a one-bump function with orbital $(x', x)$.   Furthermore, the interval $[r_x', r_x)$ is a fundamental domain of $f_x$.

We show each finitely generated subgroup of $W_C$ is a restricted wreath product of finitely many copies of $\mathbb{Z}$ with itself.  It follows that $W_C$ is locally solvable.  Let $f_{x_1}, f_{x_2}, \cdots, f_{x_n} \in W_C$, $T$ the pure tower associated to the prior elements, and $G = <f_{x_1}, f_{x_2}, \cdots, f_{x_n}>$.  Renumber the double subscripts of signatures of $T$ so they match the order of the points $x_1, \cdots, x_n$ in the unit interval.  Recall $[r_{x_i}', r_{x_i})$ is a fundamental domain of $f_{x_i}$.  Since $r_{x_i}' < x_j' < x_j < r_{x_i}$ for all $j < i$, the support of each $f_{x_j}$ is contained within a fundamental domain of $f_{x_i}$ for all $j < i$.  Therefore, $T$ is a pure fundamental tower. By \ref{stdwr}, $G \cong \mathbb{Z} \wr_n \mathbb{Z}$.
\end{proof}

Since there are uncountably many countable sets, we have produced uncountably many locally solvable subgroups in $F$.  The remainder of this paper is devoted to proving these groups are distinct when their ordered sets are.  To this end, let {\bf WC} represent the collection of groups constructed in \ref{construction}.

We also remark that there do exist uncountably many of these groups which embed in each other abstractly.  For example, there are uncountably many countable ordinals, and their corresponding groups form a chain in the same way those ordinals do.  However, our prior results illustrate that the geometric nature of PLo(I) prevents this abstract chain from embedding in the subgroup lattice of PLo(I).  This follows because the union would be an uncountable locally solvable subgroup of PLo(I).

\mysection{Non-Isomorphism Results}\mylabel{non-isomorphism}

In this section, we prove one of our main results.  

\begin{thm}
Given two countable ordered sets $C$ and $D$, the groups $W_C$ and $W_D$ as constructed in Theorem \ref{construction} are isomorphic as groups if and only if $C$ and $D$ are isomorphic as countable ordered sets.
\end{thm}

\begin{proof}
 If $C$ is isomorphic to $D$, then it's obvious that $W_C$ and $W_D$ are isomorphic.  The other direction will follow from the heavy analysis in the rest of this section.
\end{proof}

This analysis is divided into 2 main parts---1. Maps induced on towers by injective homomorphisms, and 2. Inj-isomorphisms.

Recall, since the elements of our groups are functions which act on the unit interval, we can discuss whether a relation between functions is satisfied at a point in $I$.  Let $A \subseteq I$. We say that a relation R is true {\bf on $A$} if R is true at every point in $A$.  We often use local analysis of relations to generalize results from pure to nonpure functions.

\subsection{Connecting Orbital Geometry and Algebraic Relations}\mylabel{orbgeorelns}

First, a powerful piece of background information:  If $(A,f)$ and $(A,g)$ are signed orbitals in a solvable subgroup of PLo(I), then near the ends of the orbital $A$, $f$ and $g$ are powers of a single element.  This follows from Lemma 3.12 in \cite{alg} as well as further analysis in that section.  Lemma \ref{it counts} also follows from this result.  We state a corollary of this result for our purposes here.

Let $H$ be a solvable subgroup of PLo(I) with a single orbital $A = (x,y)$. Consider the log slope homomorpism $\ds \varphi_A: H \longrightarrow \mathbb{R} \times \mathbb{R}$ defined by $h \mapsto (\ln((x)h'_+),\ln((y)h'_-))$ where $h'_+$ and $h'_-$ denote the right and left derivatives of $h$ leading from the left endpoint $x$ of $A$ and trailing into the right endpoint $y$ of $A$.  Then we have the following:

\begin{lem}\mylabel{control}
  If $H$ is a solvable subgroup of PLo(I) with a single orbital $A$, then there is an element $c \in H$ such that $H = <c, ker \varphi>$.  Furthmore, $c$ is a one-bump function with orbital $A$.
\end{lem}

The remarkable element $c$ is called a {\bf controller} of $H$. \\

Now we will collect basic facts about relations.  There are two main parts---results about pure elements and results about nonpure elements.  In each part, we first summarize the results with a table. 

Our calculations often involve taking inverses of elements in relations.  The following remark justifies this.

\begin{rmk}\mylabel{relnseq}
 Let $G$ be any group and $f,g \in G$. 
\begin{enumerate}
\item The relations $[a,b]=1$ or $[b,a]=1$ where $a \in \{f,f^{-1}\}$ and $b \in \{g,g^{-1}\}$ are equivalent.
\item The relations $[a^{b^n},a] = 1$ or $[a,a^{b^n}]=1$ where $a \in \{f,f^{-1}\}$, $b \in \{g,g^{-1}\}$, and $n \in \mathbb{Z}$ are equivalent.
\end{enumerate}
\end{rmk}

For all of the following proofs regarding both pure and nonpure elements, suppose that $f$ and $g$ both move points to the right on $A$ and $B$.  If they do not, take inverses of those that don't in each logical relation $R$ assumed to be true to get a new relation $R'$ which is true by the prior remark.
\vspace{.5cm}

{\bf Pure Elements} \\

Let $G$ be a group in $\mathcal{T}$, and let $f,g$ be pure elements of $G$. Then $f$ and $g$ have unique signed orbitals $(A,f),(B,g)$ associated to them.  Assuming the logical relation(s) in the first column, we give information about possible relationships of the signed orbitals along the corresponding row. In the top row are the 5 possible relationships for two signed orbitals in a group in $\mathcal{T}$. We sometimes refer to these as {\bf configurations} of $f,g$. Note that the $A = B$ case is separated into two pieces.  A ``c'' indicates that $f,g$ commute on the orbital, so $xfg = xgf$ for all $x \in A$.  An ``nc'' indicates the non-commuting case, so there is at least one $x \in A$ such that $xfg \neq xgf$.  

Since $f,g$ are pure, there is exactly one configuration of $f,g$.  If the configuration is impossible, we indicate with $\nexists$.  If exactly one of a few configurations must exist, we indicate with an exclusive or $\exists \underline{\vee}$. In some cases, we provide additional information as footnotes, and it is often useful to read across the entire row to understand all possibilities for a given relation.    For example, Row 1 says that $[f,g] = 1$ implies that exactly one of the configurations $A=B(c)$ and $A \cap B = \emptyset$ is possible.  Row 2 indicates that exactly one of the first 3 configurations exists for $[f,g] \neq 1$ to be true, and neither of the last 2 configurations exist.  

Note that by $A \subset B$, we mean $A$ is a proper subset of $B$.

\begin{table}[H]
     \begin{center}
     \begin{tabular}{ c || c c c c c  }
Relations & \multicolumn{5}{c}{Relationships Between Orbitals} \\ \\
& 1. $A \subset B$ & 2. $B \subset A$ & 3. $A = B(nc)$ & 4. $A = B(c)$* & 5. $A \cap B = \emptyset$ \\ 
\toprule
   $[f,g] = 1$ & $\nexists$ & $\nexists$ &  $\nexists$ & $\exists \underline{\vee}$ & $\exists \underline{\vee}$** \\  \\ \\ \\
   $[f,g] \neq 1$ & $\exists \underline{\vee}$  & $\exists \underline{\vee}$ & $\exists \underline{\vee}$ & $\nexists$ & $\nexists$ \\ \\ \\ \\
  $[f^g,f] = 1$ and & $\exists$ & $\nexists$ &  $\nexists$ & $\nexists$ & $\nexists$ \\ 
  $[f,g] \neq 1$ \\ \\
      
      \\ \bottomrule
      \end{tabular}
      \caption{Pure Functions---Relations and two orbitals  \hfill *$<f,g>\cong \mathbb{Z}$ \\ \text{ } \hfill **$<f,g> \cong \mathbb{Z} \times \mathbb{Z}$}
      \mylabel{puretable}
      \end{center}
      \end{table}

We start numbering of rows with 1.\\

We use the following lemma to prove the first 2 rows of Table \ref{puretable} in the corollaries that follow.

\begin{lem}\mylabel{commute}
 Let $f,g$ be pure elements of a group $G \in \mathcal{T}$ and $(A,f),(B,g)$ their signed orbitals. If any of the following hold, there is a point $x \in A$ where $[f,g] \neq 1$.
\begin{enumerate}[label={(\arabic*)},ref={\thecor~(\arabic*)}]
\item \mylabel{c1} $B \subset A$ (We show $fg \neq gf$ at a point in $A - B$ and thus $[f,g] \neq 1$).
\item \mylabel{c2} $A=B$ and $f=g$ near an end of $A$, but not throughout $A$
\item \mylabel{c3} $A=B$ and $f,g$ are powers of a common element $c \in$ PLo(I) near the ends of $A$ but not throughout $A$
\end{enumerate}

\end{lem}

\begin{proof} 
\ref{c1}. Let $B=(b,d)$ and $A=(a,c)$.   Since $G$ has no transition chains $A$ does not share an end with $B$.  Therefore, $g$ fixes points in $A$ which are left of $B$, i.e., $g$ fixes $[a,b]$. Pick a point $x$ in $B \cap A$.  Apply $f^{-1}$ to $x$. Since $f$ maps points in the same direction on $B$, $Bf \not\subset B$.  Since $G$ has no transition chains, $f^{-1}$ must map $B$ off itself or else the pair $(B,g)$, $(Bf,g^f)$ of signed orbitals would be a transition chain. Consider the image point $y = (x)f^{-1}$ which is in $(a,b)$.  Then $(y)fg = (x)g$ whereas $(y)gf = (y)f = x$.  Therefore $fg \neq gf$.  \\

\ref{c2}.  Suppose $f=g$ near the left end of $A=(a,c)$, an orbital which they share.  The case for the right end is similar. There is some smallest bouncepoint $b$ of the pair $f,g$, so $f=g$ on $[a,b]$. Therefore, $fg^{-1}$ has an orbital $B = (b,d)$ for some $d$ in the interval $(b,c)$. (Note: $d\neq c$ since that would result in two elements which share an end but not both, contradicting \ref{end}.) The previous part of the lemma applies to the pair $f$, $fg^{-1}$ and shows $[f,fg^{-1}] \neq 1$.  Therefore $[f,g] \neq 1$. 
\\

\ref{c3}. By assumption $f = c^{p_1}$ and $g = c^{p_2}$ near the ends of $A$ for some integer $p_i$ and some $c \in PLo(I)$.  Let $p = p_1p_2$. Let $h_1 = f^{p_2}$, $h_2 = g^{p_1}$.  It is enough to show $h_1$ and $h_2$ do not commute, since powers of commuting elements commute.  Furthermore, the functions $h_1$ and $h_2$ equal $c^p$ near the ends of $A$ and are not equal on the whole orbital since $f$ and $g$ are not powers of a common element along the whole orbital.  Hence, we can apply the previous part of the lemma to conclude $[h_1,h_2] \neq 1$.
\end{proof}

\begin{cor}\mylabel{purerow1} (Row 1) of Table 1 and footnotes * and **. \\
 Let $G$ be a group without transition chains and $(A,f), (B,g) \in \mathcal{SO}(G)$ be pure.  
 \begin{enumerate}
 \item If $[f,g] = 1$, then $A \cap B = \emptyset$ or $A = B$.  
 \item If $A=B(c)$, then $<f,g> \cong \mathbb{Z}$. 
 \item If $A \cap B = \emptyset$, then $<f,g> \cong \mathbb{Z} \times \mathbb{Z}$.
 \end{enumerate}
\end{cor}

\begin{proof}
The contrapositive of \ref{c1} is $[f,g] = 1$ implies the orbital of $g$ is not properly contained in the orbital of $f$.  Thus the (1,1)-entry is done.  To get the other case for the orbital of $f$ properly in the orbital of $g$, i.e. the (1,2)-entry, switch the roles of $f$ and $g$ in \ref{c1} to get the equivalent relation $[g,f] = 1$.  Since $[f,g] = 1$ the only remaining possibilities are $A = B (c)$ or $A \cap B = \emptyset$. 

If $A = B$, then $<f,g>$ has a single group orbital $A$.  Since $G \in \mathcal{T}$, it is locally solvable by \ref{loc solv}.  Then $<f,g>$ is solvable since it is a finitely generated subgroup of a group in $\mathcal{T}$. By Lemma \ref{control} there is a controller $c$ for $<f,g>$.  Thus the elements $f$ and $g$ are powers $f=c^{p_1}$ and $g=c^{p_2}$ of $c$ near the ends of $A$. If this is not true on all of $A$, then \ref{c3} shows $[f,g] \neq 1$ in $<f,g>$, a contradiction.  Hence $f=c^{p_1}$ and $g=c^{p_2}$ on all of $A$ for some $p_1,p_2 \in \mathbb{Z}$.  Therefore, $<f,g> \cong \mathbb{Z}$.  

If $A \cap B = \emptyset$, then $f$ and $g$ are non-trivial elements with disjoints supports.   Thus $<f,g> \cong \mathbb{Z} \times \mathbb{Z}$.
\end{proof}

\begin{cor} (Row 2) of Table 1. \\
 Let $G \in \mathcal{T}$ and $f,g \in G$ be such that $[f,g] \neq 1$.  Then $\mathcal{SO}\{f,g\}$ contains a 2-tower or $f$ and $g$ share an orbital $A$ such that $<f,g>_A \ncong \mathbb{Z}$.
\end{cor}

\begin{proof}
Assume $[f,g] \neq 1$.  From Lemma \ref{commute}, we conclude $A \subset B$, $B \subset A$, or $A=B(nc)$ are possible configurations for $A$ and $B$.  Furthermore, since $f$ and $g$ are pure, $A=B(c)$ or $A \cap B = \emptyset$ implies $[f,g] = 1$.  Thus, 1,2,3 are possible while 4,5 are not.
\end{proof}

Remark: The last two proofs result in an equivalence: $[f,g] = 1$ if and only if for all pairs of orbitals $A,B$ of $f,g$ either 4. $A \cap B = \emptyset$ or 5. $A = B$ and $<f,g> \cong \mathbb{Z}$.

\begin{lem}\mylabel{relns} (Rows 3) of Table 1.\\
 Let $G \in \mathcal{T}$ and $f,g \in G$ be pure.  If $[f^g,f] = 1$ and $[f,g] \neq 1$ then $A \subset B$.
\end{lem}

\begin{proof}
By Row 2 of Table 1, since $[f,g] \neq 1$, either $A \subset B$, $B \subset A$ or $A = B(nc)$.  If $B \subset A$ or $A = B$ with $<f,g> \ncong \mathbb{Z}$, we get contraditions as follows.  

Assume that $B \subset A$.  Then we can apply \ref{c1} to conclude that $fg \neq gf$ on $A$, that is $f^g \neq f$ on the orbital $A$. By Lemma \ref{basic conj orbitals}, $f^g$ has orbital $A$.  Thus $f^g$ and $f$ share the orbital $A$.  Since they are conjugates, Lemma \ref{conj slopes} shows they have the same leading and trailing slopes on $A$.  Thus $f^g = f$ near the ends of $A$ but not throughout $A$. Now we can use Lemma \ref{c2} to get $[f^g,f] \neq 1$, a contradiction.    
                                                                                                                                                                                                                                                                                                                                                                                                                                                                                                                                                                                                                                                                                                                                                                                                                                                                                    
Now consider the case when $A = B(nc)$.  Then $<f,g> \ncong \mathbb{Z}$ since $[f,g] \neq 1$. Since $f^g$ and $f$ share the orbital $A$ and are conjugates of one another, they are equal near the ends of $A$.  However, $f^g \neq f$ throughout $A$ since $[f,g] \neq 1$ on $A$.  Again, we can use Lemma \ref{c2} to get $[f^g,f] \neq 1$, a contradiction.
\end{proof}
\vspace{.5cm}

{\bf Nonpure Elements} \\

Table 2 summarizes the same kind of information as Table 1, but for nonpure elements.  In this case there is more subtlety because multiple configurations may exist for the same pair of functions.  Let $(A,f),(B,g)$ be nonpure signed orbitals of a group $G$ in $\mathcal{T}$. If one of the indicated configurations must exist based on the relation(s) for $f,g$, we indicate with $\exists$. If at least one must exist in a set of configurations, we indicate with ``$\exists$ or'' for each element of the set.  If a configuration violates the relation(s), we indicate with $\nexists$.  If it is possible, but not necessary, we indicate with $\checkmark$.   In some cases, we provide additional information in footnotes, and it is often useful to read across the entire row to understand the indicated statement.  For example, Row 2 can be interpreted as: If $[f,g] \neq 1$, then there exists a configuration like 1 or there exists a configuration like 2 or there exists a configuration like 3 while 4 and 5 are possible but 
not necessary. 

\begin{table}[H]
     \begin{center}
     \begin{tabular}{ c || c c c c c  }
        Relations & \multicolumn{5}{c}{Relationships Between Orbitals} \\ \\
& 1. $A \subset B$ & 2. $B \subset A$ & 3. $A = B(nc)$ & 4. $A = B(c)$* & 5. $A \cap B = \emptyset$\\
\toprule
   $[f,g] = 1$ & $\nexists$ & $\nexists$ & $\nexists$ & \checkmark & \checkmark** \\  \\ \\  \\
   $[f,g] \neq 1$ & $\exists \vee$  & $\exists \vee$ & $\exists \vee$ & \checkmark & \checkmark \\ \\ \\ \\
  $[f^{g^n},f] = 1$ and & $\exists$ & $\nexists$ & $\nexists$ & \checkmark & \checkmark \\ 
  $[f,g] \neq 1$ & \\ 
  $\text{ for some } n \in \mathbb{Z}-\{0\}$ & \\
      
      \\ \bottomrule
      \end{tabular}
      \caption{Nonpure Functions---Relations and two orbitals  \hfill *$<f,g>_A\cong \mathbb{Z}$ \\ \text{ } \hfill **$<f,g>_{A \cup B} \cong \mathbb{Z} \times \mathbb{Z}$ or $\mathbb{Z}$}
      \mylabel{nonpuretable}
      \end{center}
      \end{table}

We use the following lemma to prove Rows 1-2 of Table \ref{nonpuretable} in the following 2 corollaries.

\begin{lem}\mylabel{npcommute}
 Let $f,g$ be elements of a group $G \in \mathcal{T}$.  If any of the following hold, then $[f,g] \neq 1$.
\begin{enumerate}[label={(\arabic*)},ref={\thecor~(\arabic*)}]
\item \mylabel{npc1} An orbital of $g$ is properly contained in an orbital of $f$, or vice versa. (In this case, we show $fg \neq gf$ at a point in the larger orbital.)
\item \mylabel{npc2} $f=g$ near an end of a shared orbital $A$, but not throughout the orbital $A$. 
\item \mylabel{npc3} Each of $f$ and $g$ is some power of a common element $c \in$ PLo(I) near the ends of a shared orbital $A$ but not throughout $A$.
\end{enumerate}

\end{lem}

\begin{proof}
Given an orbital of a function, we assume the function move points to the right on the orbital whenever necessary.  Else we can take inverses and prove an equivalent relation. \\

\ref{npc1}. Let $(A,f),(C,g)$ be signed orbitals in $G$. Suppose $C \subset A$.  There may be multiple orbitals of $g$ inside of $A$, since $f,g$ aren't necessarily pure.  Apply the proof of \ref{c1}, but take $B$ to be the orbital of $g$ that is furthest left inside of $A$.  The rest of the proof is exactly the same.  If $C$ contains $A$ instead, then simply switch the roles of $f$ and $g$ in the proof and prove the equivalent relation $[g,f] \neq 1$.  \\

\ref{npc2}.  Same proof as \ref{c2}, but apply \ref{npc1} in place of \ref{c1}.

\ref{npc3}. Same proof as \ref{c3}, but apply \ref{npc2} in place of \ref{c2}.
\end{proof}

\begin{cor} (Row 1) of Table \ref{nonpuretable} and footnotes * and **. \\
 Let $G$ be a group without transition chains. If $(A,f), (B,g) \in \mathcal{SO}(G)$ then
\begin{enumerate} 
\item If $[f,g] = 1$, then $A \cap B = \emptyset$ or $A = B(c)$.
\item $A=B(c)$ implies $<f,g>_A \cong \mathbb{Z}$
\item $[f,g] = 1$ and $A\cap B = \emptyset$ implies $<f,g>_{A \cup B} \cong \mathbb{Z} \times \mathbb{Z}$ or $\mathbb{Z}$
 \end{enumerate}
\end{cor}

\begin{proof}
Assume $[f,g] = 1$. The contrapositive of \ref{npc1} is $[f,g] = 1$ implies an orbital of $g$ is not properly contained in an orbital of $f$ and vice versa, so entries (1,1) and (1,2) of Table \ref{nonpuretable} are true. The configurations which remain are for each orbital $A$ of $f$ and $B$ of $g$ we have $A = B (c)$, $A=B (nc)$, or $A \cap B = \emptyset$.  Clearly, $A=B(nc)$ is impossible since $[f,g]=1$ implies $[f,g]=1$ on all orbitals of $<f,g>$.  Hence either $A\cap B = \emptyset$ or $A=B(c)$.

If $A = B(c)$, then $[f,g] =1$ on $A$.  Apply the proof of Corollary \ref{purerow1} to get $<f,g>_A \cong \mathbb{Z}$. 

If $A \cap B = \emptyset$, then the (1,1), (1,2), and (1,3) entries imply orbitals of $f$ and $g$ that intersect $B$ and $A$, respectively, are equal to $B$ and $A$, respectively. Thus $A$ and $B$ are group orbitals of $<f,g>$. Furthermore, $f,g$ commute by assumption.  Thus the group $<f,g>_{A \cup B}$ is 2-generated, abelian, and torsion free.  By the fundamental theorem of finitely generated abelian groups, it is isomorphic to $\mathbb{Z}$ or $\mathbb{Z} \times \mathbb{Z}$.  The rank depends on how $f$ and $g$ each act simultaneously on the orbitals $A$ and $B$.  
\end{proof}

\begin{cor} (Row 2) of Table \ref{nonpuretable}. \\
 Let $G \in \mathcal{T}$ and $f,g \in G$ be such that $[f,g] \neq 1$.  Then $\mathcal{SO}\{f,g\}$ contains a 2-tower or $f$ and $g$ share an orbital $A$ such that $<f,g>_A \ncong \mathbb{Z}$.  It is also possible that configurations 4 or 5 appear in $\mathcal{SO}\{f,g\}$.
\end{cor}

\begin{proof}
For each pair of orbitals $A,B$ of $f,g$, respectively, either $A \subset B$, $B \subset A$, $A = B$, or $A \cap B = \emptyset$.  If for all pairs $A,B$, we had either $A \cap B = \emptyset$ or  $A = B$ with $<f,g>_A \cong \mathbb{Z}$, then $[f,g] = 1$.  Therefore, we must have that there exist orbitals $A,B$ of $f,g$ such that $A \subset B$, $B \subset A$, or $A = B$ with $<f,g>_A \ncong \mathbb{Z}$.

It is possible that some pairs $A,B$ satisfy $A=B(c)$ or $A \cap B = \emptyset$. The negative relation $[f,g]\neq 1$ is already satisfied thanks to the previous paragraph.  However, it is not true that $A \cap B = \emptyset$ implies $<f,g>_{A\cup B} \cong \mathbb{Z} \times \mathbb{Z}$, because there could be other orbitals of $f$ or $g$ that intersect $B$ or $A$, respectively.
\end{proof}

Remark: The last two proofs result in an equivalence: $[f,g] = 1$ if and only if for all pairs of orbitals $A,B$ of $f,g$ either 4. $A \cap B = \emptyset$ or 5. $A = B$ and $<f,g>_A \cong \mathbb{Z}$.

\begin{lem}\mylabel{relns} (Rows 3) of Table \ref{nonpuretable}.\\
 Let $G \in \mathcal{T}$ and $f,g \in G$.  If $[f^{g^n},f] = 1$ for some $n \in \mathbb{Z}-\{0\}$ and $[f,g] \neq 1$ then
\begin{enumerate}[label={(\arabic*)},ref={\thecor~(\arabic*)}]
 \item There exist orbitals $A$ of $f$ and $B$ of $g$ with $A \subset B$.\mylabel{np41}
 \item No orbital of $g$ is properly contained in an orbital of $f$. \mylabel{np42}
 \item\mylabel{np43} If $A = B$ where $A$ is an orbital of $f$ and $B$ is an orbital of $g$, then $<f,g>_A \cong \mathbb{Z}$.
\item It is possible that $A \cap B = \emptyset$ for $(A,f),(B,g) \in \mathcal{SO}(G)$.\mylabel{np44}
\end{enumerate}
\end{lem}

\begin{proof}
By Row 2 of Table \ref{nonpuretable}, since $[f,g] \neq 1$, there exist orbitals $A$ of $f$ and $B$ of $g$ such that $A \subset B$, $B \subset A$ or $A = B$.  If $B \subset A$ or if $A = B(nc)$ (i.e., $A = B$ with $<f,g>_A \ncong \mathbb{Z}$), we get contradictions as follows.  

Assume that $B \subset A$.  Then we can apply \ref{npc1} to conclude that $fg \neq gf$ on $A$, that is $f^g \neq f$ on the orbital $A$.  (This does not follow from $[f,g] \neq 1$ directly.  Now that $f,g$ are nonpure, we must use that $[f,g]\neq 1$ {\bf on A}.) By Lemma \ref{z}, $[f, g^n] \neq 1$.  By Lemma \ref{basic conj orbitals}, $f^{g^n}$ has orbital $A$.  Thus $f^{g^n}$ and $f$ share the orbital $A$.  Since they are conjugates, Lemma \ref{conj slopes} shows they have the same leading and trailing slopes on $A$.  Thus $f^{g^n} = f$ near the ends of $A$ but not throughout $A$. Now we can use Lemma \ref{npc2} to get $[f^{g^n},f] \neq 1$, a contradiction.    
                                                                                                                                                                                                                                                                                                                                                                                                                                                                                                                                                                                                                                                                                                                                                                                                                                                                                    
Now assume $A=B$ and assume toward a contradiction that $A = B(nc)$.  Since $f^{g^n}$ and $f$ share the orbital $A$ and are conjugates of one another, they are equal near the ends of $A$.  However, $f^{g^n} \neq f$ throughout $A$ since $[f,g] \neq 1$ on $A$ implies $[f,g^n] \neq 1$ on $A$ by Lemma \ref{z}.  Again, apply Lemma \ref{npc2} to get $[f^{g^n},f] \neq 1$, a contradiction.

So far we have shown $A \subset B$ or $A = B$ with $<f,g>_A \cong \mathbb{Z}$.  To prove \ref{np41}, observe that $[f,g] \neq 1$ necessitates a case like $A \subset B$ or else $f$ and $g$ would commute.

It is possible that $A=B(c)$ or $A \cap B = \emptyset$ for some signed orbitals $(A,f),(B,g)$ in $G$.  For example, if we let $f=g_1$ and $g=g_3$ in Figures \ref{split} and \ref{freecollapse}, then $[f,g] \neq 1$ and $[f^{g^n},f]=1$. In Figure \ref{split}, there is a configuration like $A\cap B = \emptyset$. In Figure \ref{freecollapse}, there is one like $A=B(c)$.
\end{proof}

\subsection{Induced Maps Between Towers}\mylabel{inducedmaps}

The previous results will be used to analyze injective homomorphisms from $\mathbb{Z} \wr_m \mathbb{Z}$ to $\mathbb{Z} \wr_n \mathbb{Z}$.  The study will
concentrate on the induced maps from towers in the domain to towers in the range.  The results that follow start very restricted and build to more general conclusions.

Analysis of maps induced on towers by homomorphisms will play a crucial role in distinguishing examples of groups from one another. Essentially, we show that towers carry group relations, and this has consequences for the induced maps which we now define.  The notation is a bit cumbersome, so think of induced maps as taking chains or parts of chains of posets to one another. An injective homomorphism $\phi: G \longrightarrow H$ between solvable groups, in particular, will carry information about the posets $\mathcal{O}(G)$ and $\mathcal{SO}(G)$ to the posets $\mathcal{O}(H)$ and $\mathcal{SO}(H)$.  This information will be very useful in distinguishing groups from each other.

Recall the ordering on the poset $\mathcal{SO}(G)$ is inherited from $\mathcal{O}(G)$ by making new chains for distinct elements with the same orbitals.  Let $G_1$ be the Hasse diagram of $\mathcal{O}(G)$ and $G_2$ be the Hasse diagram of $\mathcal{SO}(G)$.  Then $G_1$ is a simplified version of $G_2$ where repetitive information is removed. Given an element orbital $A$, there are infinitely many nodes in $G_2$ with orbital $A$.  If we identify all such nodes, we get the diagram $G_1$. Furthermore, since a maximal chain in $G_2$ is given by an ordered set $I$ of orbitals $\{O_i  \, | \, i \in I\}$ where $i < j$ in $I$ implies $O_i \subset O_j$, we have not lost any essential information in this identification.  Hence, it is often useful to envision the poset $\mathcal{O}(G)$, though many of our calculations necessitate using $\mathcal{SO}(G)$ because the functions play an important role in our calculations. 

 Let $\phi : G \longrightarrow H$ be a homormorphism between subgroups of PLo(I).  For each pair $(T,S)$ where $T$ is a tower in $G$ and $S$ is a tower in $\phi(\mathcal{S}(T))$, there is an {\bf induced map} $\ds \hat{\phi}_{(T,S)} : T \longrightarrow S \cup \{\emptyset\}$ which takes a pair $(B,g) \in T$ to either 1. the pair $(C, \phi(g)) \in S$ where C is the orbital of $\phi(g)$ in $S$; or 2. $\emptyset$ if $\phi(g)$ is not a signature appearing in $S$.  This map is well defined by the definition of a tower.  Note that the existence of $\hat{\phi}$ also implies the existence of an obvious unique induced map between the associated stacks.  However, since infinitely many signatures share the same orbitals, an induced map of stacks corresponds to infinitely many induced maps on towers.

Consider the following figures from Section \ref{representations}:

\begin{figure}[H]
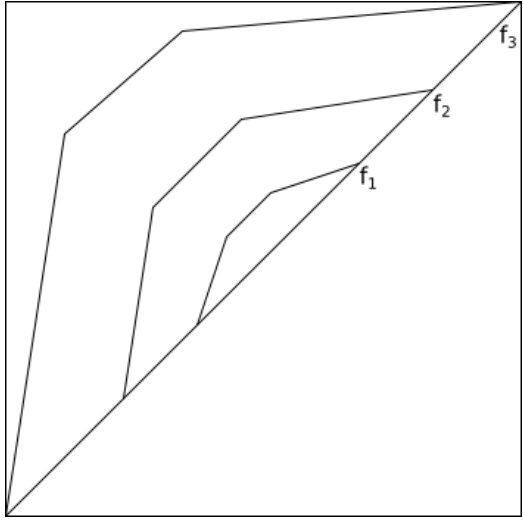
\caption{Standard}\mylabel{w32}
\figw
\end{figure}

\begin{figure}[H]\caption{Split}\mylabel{split2}
\figsplit
\end{figure}

 and the isomorphism $\phi: W_3 \longrightarrow G$ defined by $\phi(f_i) = g_i$ where $W_3$ is represented by the group in Figure \ref{w32} and $G$ is the group generated by the elements in Figure \ref{split2}.  Let the maximal tower in Figure \ref{w32} be $T$, the left tower in Figure \ref{split2} be $M_1$, and the right one $M_2$.    The induced map $\p{T}{M_2}$ illustrates the necessity of the dummy image element $\emptyset$.  The maps $\p{T}{M_1}$ and $\p{T}{M_2}$ illustrate dependence on the codomain tower and the maps $\hat{\phi^{-1}}_{(M_1,T)}$ and $\hat{\phi^{-1}}_{(M_1,T)}$ illustrate dependence on the choice of domain tower.
 
Occasionally, we omit the subscript $(T,S)$ of $\p{T}{S}$ once the induced map is clear.  We also omit $\{\emptyset\}$ or $\emptyset$ in expressions involving induced maps. Since $S$ is a subtower of $\phi(\mathcal{S}(T))$ and since a signature can only appear once in any given tower, card($T$) $\geq$ card($S$).   Finally, if $T$ is a pure tower and $S$ is any tower in $\phi(T)$, the preimage $\hat{\phi}^{-1}(S)$ is a subtower of $T$.  
 
{\bf Important}: The notation $\phi(T)$ is used to denote $\mathcal{SO}(\phi(\mathcal{S}(T)))$, that is, the set of all signed orbitals of images of signatures in $T$.  Note that $\phi(T)$ is a poset, but it may not form a tower.

An induced map $\ds \p{T}{M}$ is {\bf maximal} if $M$ is maximal in $\phi(T)$.  The example $\p{T}{M_2}$ from is a maximal induced map.  A non-maximal induced map can be obtained by mapping $g_1$ or $g_3$ to $\emptyset$.  

We can always extend an induced map $\p{T}{S}$ to a maximal induced map $\p{T}{M}$ by taking a maximal tower $M$ in $\phi(T)$ which contains $S$.  Then we define $\p{T}{M}$ by the same process we performed to define $\p{T}{S}$. 

The nature of the poset $\phi(T)$ is a question we will consider in detail when $\phi$ is injective and the groups are solvable.  \\ \\

The results which follow build up a basic understanding of induced maps between solvable groups.  Eventually we will show there is some rigidity of towers under injective homormorphisms and that all possible maximal induced maps can be well understood if the domain is a nice enough tower.  First, we give some background and make helpful observations about group relations of groups in $\mathcal{T}$.  Group relations restrict how orbitals relate, and this has some very important consequences for induced maps between towers. \\

For the proof of the following, recall if $T$ is a tower in $G$, then the {\bf type} of the tower T denoted type(T) is the order type of the ordered set of right endpoints of the orbitals in $\mathcal{O}(T)$.  Furthermore, if $T$ is a tower of type $n$, we call refer to $T$ as an {\bf n-tower}.

\begin{prop}\mylabel{puretowers}
Let $G, H$ be solvable subgroups of $PLo(I)$ and $\phi:G \longrightarrow H$ an injective homomorphism.  If $T$ is a pure tower of $G$ such that $\phi(T)$ is pure, then
\begin{enumerate}
 \item The set $\phi(T)$ is also a tower. 
\item The induced map $\p{T}{\phi(T)}$ is an isomorphism of ordered sets. 
\end{enumerate}
\end{prop}

\begin{proof}
 The proof is by induction on the tower.  The tower type of any tower in $G$ or $H$ is a natural number by Theorem \ref{thm geo}.  If $T$ is a tower of type 1, then the result is obvious.  Suppose $T = \{(A_1, f_1),(A_2, f_2), \dots ,(A_n, f_n)\}$ is a tower of type $n$ where $i < j$ implies $A_i \subset A_j$.  Let $g_i = \phi( f_i )$, $1 \leq i \leq n$.  Existence of a tower of type $n$ in the image subgroup $<g_1, g_2, \cdots, g_n>$ is guaranteed by Theorem \ref{thm geo}. We still need to show that the images of the signatures of $T$ form an $n$-tower and that $\phi$ respects the ordering on the towers. Note that since $T$ has $n$ elements, $\phi(T)$ is pure, and $\phi$ is injective, we know $\phi(T)$ also contains $n$ elements.

We prove the $n=2$ case separately because the general case uses 2-towers.  Use the same notation as that in the first paragraph.  That is, let $T = \{(A_1, f_1),(A_2, f_2)\}$ be any pure tower of height 2 in $G$ such that $A_1 \subset A_2$, and let $g_i = \phi (f_i)$. Since $\phi$ is an isomorphism on its image, relations in $<f_1,f_2>$ hold if and only if they hold in $<g_1,g_2>$. By Lemma \ref{c1}, $[f_1,f_2] \neq 1$. Since $f_2$ conjugates the orbital $A_1$ of $f_1$ off itself, we also know $[f_1^{f_2}, f_1] = 1$.  Thus $[g_1,g_2] \neq 1$ and $[g_1^{g_2},g_1] = 1$.  By row 3 of table 1, there exists an orbital $B_1$ of $g_1$ contained in an orbital $B_2$ of $g_2$.  Since $\phi(T)$ is pure, these are the unique orbitals of $g_1$ and $g_2$.  Thus $\phi(T) = \{(B_1,g_1),(B_2,g_2)\}$ is a 2-tower, and $\p{T}{\phi(T)}$ is an isomorphism of ordered sets.
 
 Assume inductively that the result holds for towers of type $k$ where $k < n$.  Let $T$ be a pure $n$-tower $T = \{(A_1, f_1),(A_2, f_2), \dots ,(A_n, f_n)\}$ such that $\phi(T)$ is a pure. Again let $\phi(f_i) = g_i$ and let $\phi(T) = \{(B_1,g_1), \cdots, (B_n,g_n)\}$---not necessarily in order from smallest to largest. Let $P$ be the set of all towers in $\phi(T)$.   Each tower in $P$ is pure because $\phi(T)$ is pure. Consider the $2-$ and $(n-1)-$ pure subtowers $U = \{(A_{n-1},f_{n-1}),(A_n,f_n)\}$ and $V = \{(A_1, f_1),(A_2, f_2), \dots ,(A_{n-1},f_{n-1})\}$ of $T$.  They are pure because they are subtowers of the pure tower $T$.  Furthermore, $\phi(U)$ and $\phi(V)$ are subsets of the pure poset $\phi(T)$, hence are pure.  By the previous proof for 2-towers and by the inductive hypthotheses, $\phi(U)$ and $\phi(V)$ are towers. Also, the induced maps $\ds \hat{\phi}_{(U,\phi(U))}$ and $\hat{\phi}_{(V,\phi(V))}$ are isomorphisms of ordered sets. Notice, these are just restrictions of the induced map $\hat{\phi}_{(T,\phi(T))}$ to certain subtowers of $T$. Furthermore, 
$U \cup V = T$ and $U \cap V = (A_{n-1},f_{n-1})$.  Since the orbital of $\hat{\phi}_{(U,\phi(U))}(A_n,f_n)$ properly contains the orbital of $\hat{\phi}_{(U,\phi(U))}(A_{n-1},f_{n-1})$, it will also properly contain all the other orbitals in the image of $\hat{\phi}_{(V,\phi(V))}$ thanks to $\phi(T)$ being a pure poset.  Thus $\phi(T)$ is a tower and the map $\hat{\phi}_{(T,\phi(T))}$ is an isomorphism of ordered sets. 
\end{proof}

The following expands the statements of Proposition \ref{puretowers} to a new Proposition where only the domain tower is required to be pure.

\begin{prop}\mylabel{puretower}
 Let $\phi: G \longrightarrow H$ be an injective homomorphism between solvable subgroups of PLo(I) and let $T$ be a pure $n$-tower in $G$.  
\begin{enumerate}
 \item There exists an $n$-tower in $\phi(T)$. 
\item Any induced map $\p{T}{U}$ is non-decreasing on elements not mapped to $\emptyset$.
\item There exists an induced map $\p{T}{S}$ which is an isomorphism of ordered sets.
\end{enumerate}
\end{prop}

\begin{proof}  
We prove this by induction on the type of $T$.  If $T = \{(A_1,f_1)\}$ is a tower of type 1, then since $\phi$ is injective, $f_1$ maps to a non-trivial element of $H$.  Since $f_1$ is non-trivial, it has an orbital.  Thus, there is a 1-tower in $\phi(T)$, and the induced map is an isomorphism of order sets.  Since there are only 1-towers in $\phi(T)$, the second part holds trivially for a tower of type 1.

Let $T = \{(A_1,f_1),(A_2,f_2)\}$ be a tower of type 2 and $\phi(f_i) = g_i$.  Since $T$ is pure, $[f_1^{f_2},f_1] = 1$.  Also, $[f_1,f_2] \neq 1$  by Lemma \ref{c1}.  Thus $[g_1^{g_2},g_1] = 1$ and $[g_1,g_2] \neq 1$.  By row 3 of table 2, there is an orbital of $g_1$ properly contained in an orbital of $g_2$, and no orbital of $g_2$ is contained in an orbital of $g_1$.   Thus, there is a 2-tower $S$ in $\phi(T)$ such that $\p{T}{S}$ is an isomorphism of ordered sets.  Also, any induced map $\p{T}{U}$ is non-decreasing on elements not mapped to $\emptyset$ since no orbital of $g_2$ can be contained in an orbital of $g_1$.

If $U$ is any tower in $\phi(T)$, consider the induced map $\p{T}{U}$. To complete the second part of the proposition, it is enough to show that the maximal induced map $\p{T}{M}$ for any maximal tower $M$ of type $3$ or higher in $\phi(T)$ containing $U$ preserves order on elements not mapped to $\emptyset$.  Extend $\p{T}{U}$ to a maximal induced map $\p{T}{M}$ where $M = \{(B_{11},g_{11}),(B_{12},g_{12}), \cdots, (B_{1m},g_{1m})\}$ is maximal in $\phi(T)$ and contains $U$.  Now ignore elements mapped to $\emptyset$ by restricting the domain of $\p{T}{M}$ to the subtower $S = T - ({\p{T}{U}}^{-1}(\emptyset))$. For convenience, renumber the subscripts of signatures of $S$ consecutively to $S = \{(A_1, f_1),(A_2, f_2), \dots ,(A_k, f_k)\}$ where $0 \leq k \leq m$ and such that $i < j$ implies $A_i \subset A_j$.  Also, rename the $g_{jl}$ so that $\phi(f_i) = g_i$.  Then, whenever $i < j$, $[f_i, f_j] \neq 1$ and $[f_i^{f_j}, f_i] = 1$ because $S$ is pure.  Since $\phi$ is an isomorphism, these relations hold when $f$ is replaced by $g$.  Row 3 of Table 2 implies $B_i \subset B_j$ or $B_i = B_j$.  Hence $\p{S}{M}$ is non-decreasing, so $\p{T}{U}$ is non-decreasing on elements not mapped to $\emptyset$.

Let $T = \{(A_1, f_1),(A_2, f_2), \dots ,(A_n, f_n)\}$ in order of containment of orbitals and $n \geq 3$. Assume the proposition holds for all $k < n$. Let $\phi(f_i) = g_i$.  Consider the subtower $T' = \{(A_1, f_1),(A_2, f_2), \dots ,(A_{n-1}, f_{n-1})\}$ and note $n-1 \geq 2$.  By the inductive hypotheses, there is an $n-1$ tower in $\phi(T')$.  Let $V$ be an arbitrary $(n-1)$-tower in $\phi(T')$ and note that each signature $g_i, 1 \leq i \leq n-1$ appears exactly once in $V$.  Also by the inductive hypotheses, the map $\p{T'}{V}$ is nondecreasing on elements not mapped to $\emptyset$.  If $g_n$ has an orbital that contains the largest orbital of $V$ properly, then $\phi(T)$ contains an $n$-tower.  If $g_n$ does not have an orbital which properly contains the largest orbital of $V$, we get a contradiction of relations as follows. Let $i<j<k$ and consider the relations $[f_{i}^{f_k},f_{j}] = 1$.  These relations also hold when $f$ is replaced by $g$.  In particular, observe that $[f_{n-1}^{f_n},f_{n-2}] = 1$.  Since no orbital of $g_n$ contains the largest orbital of $V$, we have $g_{n-1}^{g_n} = g_{n-1}$ on $\mathcal{O}(V)$.  Therefore, $[g_{n-1}^{g_n}, g_{n-2}] = [g_{n-1},g_{n-2}] \neq 1$ on $\mathcal{O}(V)$ which implies the negative relation also holds in the group $H$, a contradiction. Thus $g_n$ has an orbital $B_n$ which contains the largest orbital of $V$ properly and the map $\p{T}{V\cup\{(B_n,g_n)\}}$ is an isomorphism of ordered sets.

\end{proof}

We need our results to be a bit more general later.  We define a pair of elements $f,g \in$ PLo(I) to be {\bf untwisted} if for every pair of signed orbitals $(A,f), (B,g) \in \mathcal{SO}\{f,g\}$ such that $[f,g] \neq 1$ at some point in $A \cap B$, the same configuration holds on ALL other orbitals of $f,g$ on which there is a point $x$ such that $(x)[f,g] \neq 1$.  We define a poset  $P$ of signed orbitals to be {\bf untwisted} if for every pair of elements $(A,f), (B,g) \in P$, the pair $f,g$ is untwisted.  More specifically, suppose $(C, f)$, $(D, g)$ are other signed orbitals that intersect and $[f, g] \neq 1$ at a point in $C \cap D$.  If $B \subset A$, then $D \subset C$.   If $A \subset B$, then $C \subset D$. If $A = B(nc)$, then $C = D (nc)$.   

Thus a tower T is {\bf untwisted} if for every pair of signed orbitals $(A,f),(B,g) \in T$ with $(A,f) < (B,g)$, the same relationship holds for ANY orbitals of $f$ and $g$ on which they do not commute.    The purpose of untwisted towers is if $(A, f) < (B, g)$ in an untwisted tower $T$, then the relations $[f, g] \neq 1, [f^{g^n}, f] = 1$ for some $n \in \mathbb{Z}^+$ hold in $<f,g>$.

With this new terminology, we have a Scholium which generalizes Proposition \ref{puretower} to induced maps with an untwisted domain tower.

\begin{sch}\mylabel{OrderUntwist}
 Let $\phi: G \longrightarrow H$ be an injective homomorphism between solvable subgroups of PLo(I) and let $T$ be a untwisted $n$-tower in $G$.  
\begin{enumerate}
 \item There exists an $n$-tower in $\phi(T)$. 
\item Any induced map $\p{T}{U}$ is non-decreasing on elements not mapped to $\emptyset$.
\item There exists an induced map $\p{T}{S}$ which is an isomorphism of ordered sets.
\end{enumerate}
\end{sch}

\begin{proof}
Consider a maximal induced map $\p{T}{M}$ where $T$ is untwisted.  The proof follows in the same way as proposition \ref{puretower} with untwisted replacing pure, relations $[f_i^{f_j^n}, f_i]$ for some $n \in \mathbb{Z}^+$ replacing relations of the form $[f_i^{f_j},f_i]$ for $j > i$, and Row 3 of Table 2 replacing use of Row 3 of Table 2.  
\end{proof}

\begin{lem}\mylabel{injective}
Let $\phi: G \longrightarrow H$ be an injective homomorphism of solvable subgroups and $T$ an untwisted tower in $G$.  If $S$ is any tower in $\phi(T)$, then $\p{T}{S}$ is injective except possibly on elements that map to $\emptyset$ or the minimum of $S$.
\end{lem}

\begin{proof}
Suppose $T$ is untwisted.  Let $(B_1, g_1)$ be the minimal element of $S$.  Suppose toward a contradiction that some non-minimal element $(B_i, g_i) \in S$ has two distinct elements $(A_i, f_i) < (A_j, f_j) \in T$ mapped to it under $\ph$.  Since $S \subset \phi(T)$, there exists some $(A_1, f_1) \in T$ such that $\ph(A_1, f_1) = (B_1, g_1)$.  Since $\ph$ is nondecreasing on elements not mapped to $\emptyset$ and $(B_1, g_1) < (B_i, g_i)$, we have $(A_1, f_1) < (A_i, f_i)$ and $(A_1, f_1) < (A_j, f_j)$.  Since every subtower of an untwisted tower is untwisted, $\{(A_1, f_1), (A_i, f_i), (A_j, f_j)\}$ is a untwisted 3-tower.  By \ref{moveover} and \ref{basic conj orbitals}, there exists a $n \in \mathbb{N}$ such that $[f_1^{f_j^n}, f_i] = 1$.  Thus the relation $[g_1^{g_i^n}, g_i]=1$ holds in $\phi(G)$.  However, the signed orbitals of the elements $g_1^{g_i^n}$ and $g_i$ form a 2-tower $\{(B_1, g_1), (B_i, g_i)\}^{g_i^n}$ by \ref{basic conj towers}.  Thus, $[g_1^{g_i^n}, g_i] \neq 1$ by Lemma \ref{npc1}.  This is a contradiction.  Therefore a unique element of $T$ maps to each non-minimal element of $S$.
\end{proof}

To see examples of maps induced on towers which map multiple elements to $\emptyset$ and to the minimum of $S$, refer to the homomorphisms defined in Section \ref{representations} using Figures \ref{w3}, \ref{free}, and \ref{freecollapse}. Set $i = 1$. Let $T$ be the tower of generators in Figure \ref{w3} and let $M$ and $M'$ be the maximal towers on the right of Figures \ref{free} and \ref{freecollapse}, respectively.  Then $\p{T}{M}$ maps the bumps of $f_2$ and $f_3$ to $\emptyset$ and $\p{T}{M'}$ maps all three bumps of $f_1, f_2, f_3$ to the single bump in $M'$.

\begin{cor}\mylabel{strictly increasing}
Let $\phi: G \longrightarrow H$ be an injective homomorphism of solvable subgroups and $T$ a pure (more generally, untwisted) tower in $G$.  If $S$ is any tower in $\phi(T)$, then $\p{T}{S}$ is strictly increasing except possibly on elements that map to $\emptyset$ or the minimum of $S$.
\end{cor}

\begin{lem}\mylabel{untwisted image}
If $\phi: G \longrightarrow H$ is an injective homomorphism between solvable subgroups generated by pure towers $T$ and $S$, respectively.  Let $U$ be any untwisted tower in $G$, then $\phi(U)$ is untwisted.
\end{lem}

\begin{proof}
Supposed $U$ is an untwisted tower in $G$ and let $(A,f), (B,g)$ be elements of $U$ such that $(A, f) < (B, g)$ in $U$.  Then $[f^{g^n}, f] = 1$ for some $n \in \mathbb{Z}^+$.  Furthermore, since $f,g$ have orbitals forming a 2-tower, $[f,g] \neq 1$ by \ref{npc1}.   Since $\phi$ is an injective homomorphism, the relations $[\phi(f)^{\phi(g)^n}, \phi(f)] = 1$ and $[\phi(f), \phi(g)] \neq 1$ hold in $\phi(G)$.  By Row 3 of Table 2, given orbitals $C, D$ of $\phi(f), \phi(g)$, respectively, the only possible configurations are $C \subset D$, $C = D (c)$, and $C \cap D = \emptyset$.  The last 2 possibilities are ones in which $\phi(f)$ and $\phi(g)$ commute with respect to the orbitals, so the only noncommuting possibility is an orbital of $\phi(f)$ is properly contained in an orbital of $\phi(g)$.  Thus $\phi(T)$ is untwisted.
\end{proof}

\begin{lem}\mylabel{2towerexpand} (Expansion up of maps onto 2-towers)

Let $G, H$ be solvable subgroups of $PLo(I)$, $\phi:G \longrightarrow H$ an injective homomorphism, and $T$ be a pure tower of $G$.  If $S = \{(B_1,g_1), (B_2,g_2)\}$ is any 2-tower in $\phi(T)$, then the induced map from $\ph^{-1}(B_2, g_2) \cup \uparrow \ph^{-1}(B_2,g_2) \longrightarrow (B_2,g_2) \cup \uparrow S$  where $\uparrow$ is taken with respect to the posets $T$ and $\phi(T)$, respectively, is an isomorphism of ordered sets.
\end{lem}

\begin{proof}
Note the domain $\ph^{-1}(B_2, g_2) \cup \uparrow \ph^{-1}(B_2,g_2)$ is pure because it's a subtower of $T$.  By Lemma \ref{injective}, $\ph: (B_2, g_2) \cup \uparrow \ph^{-1}(B_2,g_2) \longrightarrow (B_2,g_2) \cup \uparrow S$ is injective on elements not mapped to $\emptyset$.  By Corollary \ref{strictly increasing}, it is strictly increasing on elements not mapped to $\emptyset$.    We still need to show the map does not send any elements of the domain to $\emptyset$, i.e., the map sends each element of $\ph^{-1}(B_2,g_2) \uparrow \ph^{-1}(B_2,g_2)$ to an element of $(B_2,g_2) \cup \uparrow S$.  Surjectivity follows easily.

Since it is the inverse image of $(B_2,g_2)$, the element $\ph^{-1}(B_2,g_2)$ does not map to $\emptyset$. Let $(A_i, f_i) \in \uparrow \ph^{-1}(B_2,g_2)$.  If $\ph(A_i, f_i) \in S$, we are done, so assume it is not.  Then there exist distinct signed orbitals $(A_1, f_1), (A_2, f_2) \in T$ which map to $S$, say, $\ph(A_i, f_i) = (B_i, g_i)$ for $1 \leq i \leq 2$.  Since $\ph$ is strictly increasing on elements not mapped to $\emptyset$, $(A_1, f_1) < (A_2, f_2)$.  Furthermore, since $(A_i, f_i) \in \uparrow \ph^{-1}(B_2,g_2)$, we have that $(A_i, f_i) > (A_2, f_2)$.  Thus $\{(A_1, f_1), (A_2, f_2), (A_i, f_i)\}$ is a pure 3-tower, so Corollary \ref{wreathrelns} shows $[f_1^{f_i}, f_2] = 1$.  Let $\phi(f_i) = g_i$.  If the orbital $B_2$ of $g_2$ is not contained in some orbital of $g_i$, then $g_1^{g_i} = g_1$ on $B_2$.  Therefore, $\phi[f_1^{f_i}, f_2] = [g_1^{g_i}, g_2] = [g_1, g_2] \neq 1$ by Lemma \ref{npc1}, a contradiction.  Thus there exists some orbital $B_i$ of $g_i$ containing $B_2$,  so $\ph(A_i, f_i) = (B_i, g_i)$.  Note also that $(B_i, g_i) \in (B_2,g_2) \cup \uparrow S$.  Thus $\ph$ is defined on all of $\uparrow \ph^{-1}(B_2,g_2)$.  
\end{proof}

\begin{cor}\mylabel{maptypes}Let $\phi: G \longrightarrow H$ be an injective homomorphism such that $G$ is solvable and generated by a pure tower.  Let $T = \{(A_1,f_1), \cdots, (A_n, f_n)\}$ be a pure nonempty tower in $G$ and let $S=\{(B_1,g_1), \cdots, (B_m, g_m)\}$ be a maximal nonempty tower in $\phi(T)$.  The induced map $\p{T}{S}$ falls into one of four categories.  The categories are mutually exclusive with the exception of 1 and 4 coinciding when $n=1$.
\begin{enumerate}
 \item Free---some nonempty subset $F = \{(A_{i_1},f_{i_1}), \cdots (A_{i_k},f_{i_k})\}$ of $T$ maps to a single signed orbital. That is, $m = 1$, $\ph(A_{i_j}, f_{i_j}) = (B_1, g_1)$ for all $1 \leq j \leq k$ and $\hat{\phi}(A_i, f_i) = \emptyset$ for all $(A_i, f_i) \in T - F$.
\item Top---some bottom portion of $T$ is mapped to $\emptyset$, some number that follow are mapped to $(B_1, g_1)$, and the rest are mapped isomorphically onto $S-(B_1,g_1)$.  That is, there exists $k, l$ with $1 < l < k < n$ such that  $\hat{\phi}(A_j,f_j) = \emptyset$ for $j \leq l$, $\ph(A_j, f_j) = (B_1, g_1)$ for $l \leq j \leq k$, and $\hat{\phi}(A_{k+i},f_{k+i}) = (B_i,f_i)$ for $2 \leq i \leq n-k = m$. 
\item Split---a combination of a free map and a top map.  In other words, there exist $j,k, l$ with $1 \leq j < k < l 
\leq n-1$ such that  $\ph(A_i,f_i) = \emptyset$ for $i < j$, $\ph(A_i,f_i) = (B_1,g_1)$ for $j \leq i < k$, $\ph(A_i,f_i) = \emptyset$ for $k \leq i < l$, and $\ph(A_{l+i-1},f_{l+i-1}) = (B_i,f_i)$ for $2 \leq i \leq m$.
\item Full---the tower $T$ itself maps order isomorphically onto $S$.  That is $n=m$ and $\ph(A_i,f_i) = (B_i,g_i)$ for $1 \leq i \leq n$.
\end{enumerate}
\end{cor}

We may also say a subtower of $T$ or $S$ is {\bf free}, {\bf top}, {\bf split}, or {\bf full} to indicate that it is the domain or image of such a maximal induced map.

\begin{proof}
Existence of each of the 4 kinds of induced maps is a simple consequence of the representations introduced in Section \ref{representations}.

That these are the only possible induced maps is a consequence of Lemma \ref{2towerexpand}.  Some element $(A_i,f_i) \in T$ maps to $(B_1,g_1)$ under $\ph$ because $S \subset \phi(T)$. If $|S| = 1$, then $\ph$ is free.

If $|S| > 1$ and $n = m$, then since each element of $S$ is in $\phi(T)$ and a signature cannot appear twice in a tower, $\ph$ is a full map.  If $n > m$, then there exists some subset $\{(A_{i_1},f_{i_1}), \cdots (A_{i_k},f_{i_k})\}$ of $T$ which maps to $\emptyset$.  Since $|S| > 1$, there exists an $(A_j,f_j) \in T, j \neq i$ which $\ph$ maps to $(B_2,g_2) \in S$.   Since $\ph$ is increasing on elements not mapped to $\emptyset$, $j > i$.  By Lemma \ref{2towerexpand} the induced map from $\{(A_i,f_i), (A_j, f_j)\}$ onto $\{(B_1,g_1),(B_2,g_2)\}$ extends.  Thus $\{(A_j,f_j), (A_{j+1},f_{j+1}), \cdots (A_n,f_n)\}$ maps order isomorphically onto $S-\{(B_1,g_1)\}$ since $S$ is maximal.  If $j = i+1$, then $\ph$ is a top map.  If $j > i+1$, then $\ph$ is a split map.  The proof is complete after recalling the $n < m$ case is impossible by the definition of induced maps.
\end{proof}

\begin{cor}Let $\phi: G \longrightarrow H$ be an injective homomorphism of solvable subgroups and $T$ a pure tower in $G$.  A maximal pure tower $S$ in $\phi(T)$ is of the same height as $T$ and order isomorphic to $T$ with order isomorphism induced by $\phi$ if and only if the bottom two bumps in $S$ form a two tower order isomorphic under $\phi$ to the bottom two bumps in $T$.
\end{cor}

\begin{lem}\mylabel{controllers}
Assume $G$ and $H$ are one-orbital solvable subgroups of PLo(I) with orbitals $A$ and $B$, respectively.  Let $c$ be a controller for $G$ and $\phi: G \longrightarrow H$ an isomorphism.  Then $\phi(c)$ is a controller for $H$.
\end{lem}

\begin{proof}
Let $c$ be a controller for $G$ on $A$  and recall that $c$ has orbital $A$, too.  If $c$ does not move points to the right, replace it with its inverse.   Recall that the controller-form of an element decomposes it into a controller part times an element in the kernel of a log slope homomorphism $\varphi$.  Let $h \in H$.  There is some element $g \in G$ such that $\phi(g) = h$.  Also, $g = c^pk$ for some $p \in \mathbb{Z}$ and some $k$ in the kernel of the log slope homomorphism on the ends of $A$.  Then, $\phi(g) = h = \phi(c)^p \phi(k)$. We need to show $\phi(c)$ has orbital $B$ and $\phi(k)$ is in the kernel of the log slope homomorphism on the ends of $B$. Let $D$ be the orbital of $k$ that is furthest left in $A$.  Every orbital of $k$ is contained in the orbital $A$ of $c$.  Thus any set $\{(-,k), (A,c)\}$ where ``-'' can be replaced by any orbital of $k$, is a 2-tower written in order.  Since $k$ has finitely many bumps, there exists an $n \in \mathbb{N}$ such that $D^{c^n}\cup\text{Supp}(k) = \emptyset$.  Therefore, $[k^{c^n}, k ] = 1$ and $[k, c] \neq 1$. Since the image must satisfy the same relations, by Table 2 Row 3 the image of each induced map from $\{(-,k), (A,c)\}$ is either a single bump or a 2-tower with the same order.  The image $<\phi(k),\phi(c)>$ has depth 2 because $<k,c>$ does.  Thus every 2-tower constructed in this way has $B$ as its largest orbital. Hence $\phi(k)$ is in the kernel of the log slope homomorphism on $B$, and $\phi(c)$ is indeed a controller.
\end{proof}

The next lemmas consider the cases that arise when we are working with elements that do not form a tower.

\begin{lem}\mylabel{nontower pure} Let $G$ and $H$ be subgroups of of PLo(I) with $G$ solvable. Let $(A,f), (B,g) \in G$ be pure. If $A = B (nc)$, $\phi: G \longrightarrow H$ is an injective homomorphism, and $\phi(f)$ and $\phi(g)$ are pure with orbitals $C$ and $D$ respectively, then $C = D (nc)$. 
 \end{lem}
 
 \begin{proof}
Since $\phi$ is injective and $[f, g] \neq 1$, we have $[\phi(f), \phi(g)] \neq 1$.  By Table 1, $\phi(f)$ and $\phi(g)$ form a 2 tower or $C = D (nc)$. By assumption $[f,g] \neq 1$, and therefore $f^g \neq f$.  However, $f^g$ and $f$ are equal on the ends of $B$ but not throughout.  Thus, $[f^g, f] \neq 1$ by \ref{npc3}. Since $\phi$ is injective, it preserves this relation.  Thus $C$ is not properly contained in $D$ or vice versa.
\end{proof}
 
 \begin{lem}Let $G$ be a solvable subgroup of PLo(I) and $f, g \in G$.
 \begin{enumerate}[label={(\arabic*)},ref={\thecor~(\arabic*)}]
 \item\mylabel{z} Then, $[f,g^n] = 1$ if and only if $[f,g] = 1$.  
 \item\mylabel{nontower}  Suppose $\phi: G \longrightarrow H$ is an injective homomorphism, $A$ is an orbital of $f$, $B$ is an orbital of $g$, $A=B(nc)$, and $\mathcal{SO}\{\phi(f),\phi(g)\}$ is untwisted.  Then $C = D (nc)$ for all orbitals $C, D$ of $\phi(f), \phi(g)$ on which they do not commute.  Furthermore, there exists at least one such configuration.
 \end{enumerate}
 \end{lem}
 
 \begin{proof}
To prove (1), let $(A,f), (B,g)$ be signed orbitals in $G$ and assume $[f, g^n] = 1$.  By Table 2 Row 1, either $A \cap B = \emptyset$ or $A = B(c)$ (with respect to $f, g^n$).  Note that $g$ and $g^n$ have the same support.  If $A = B(c)$, then $[f, g^n] = 1$ implies $<f,g^n>_A \cong \mathbb{Z}$ by footnote * of Table 2.  Thus, $<f,g^n>_A = <c>$ for some $c \in <f,g^n>_A$.  Then, $<f,g>_A$ is a subgroup of $<\sqrt[n]{c}> \cong \mathbb{Z}$.  Since every subgroup of $\mathbb{Z}$ is isomorphic to $\mathbb{Z}$, we have shown that $<f,g>_A \cong \mathbb{Z}$ and thus $[f,g] = 1$ on $A$.  The same argument can be repeated on any pairs of orbitals of $f$ and $g$ which intersect, and $f$ and $g$ certainly commute where their supports are disjoint. Hence, $[f,g] = 1$ globally.
 
 It's obvious that $[f,g] = 1$ implies $[f, g^n] = 1$. 
 
To prove (2), note that since $A = B (nc)$, we have $[f,g] \neq 1$ by assumption.  Since $\phi$ is injective, $[\phi(f), \phi(g)] \neq 1$.  Consider Row 2 of Table 2 and possible non-commuting configurations.  We wish to show that there are no 2-towers in $P = \mathcal{SO}\{\phi(f), \phi(g)\}$.  Suppose $\{(C, \phi(f)), (D_1, \phi(g))\}$ forms a 2-tower in $\phi(G)$ and $C \subset D_1$.  The case where $D_1 \subset C$ is similar. Since $P$ is untwisted, the same relationship holds on all orbitals of $\phi(f), \phi(g)$ where they do not commute.  

There exists an $n_1 \in \mathbb{Z}^+$ such that $[\phi(f)^{\phi(g)^{n_1}}, \phi(f)] = 1$ on $D_1$ since $f$ has only finitely many bumps and thus there is some positive power that maps the support of $f$ in $D_1$ off itself.    For each orbital $D_1, D_2, \cdots D_k$ of $\phi(g)$ which properly contains an orbital of $\phi(f)$, let $n_1, n_2, \cdots, n_k$ be positive integers such that $\phi([f^{g^{n_i}}, f]) = 1$ on $D_i$ for $1 \leq i \leq k$. Let $n = n_1\cdots n_k$. Then $R:= \phi([f^{g^n}, f]) = 1$ on $D_i$ for $1 \leq i \leq k$.  Furthermore, because $P$ is untwisted $\phi([f, g]) = 1$ elsewhere and thus $R = 1$ globally in $H$.  By injectivity, $\phi^{-1}(R) = [f^{g^n},f] = 1$.  By assumption $[f,g] \neq 1$ on $B$, so by \ref{z}, $[f, g^n] \neq 1$ on $B$.  Thus, $f^{g^n} \neq f$.  However, $f^{g^n}$ has the same slopes as $f$ on the ends of $B$.  So $f^{g^n}$ and $f$ are equal on the ends of $B$ but not throughout.  Thus $[f^{g^n}, f] \neq 1$ on $B$ by \ref{npc3}, a contradiction.
 
 Since $\phi$ is injective and $A = B(nc)$, we have that $[\phi(f), \phi(g)] \neq 1$.  Thus, there exists a configuration on which the elements do not commute.  By the prior paragraph, there are no two towers, so there must exist orbitals $C, D$ of $\phi(f), \phi(g)$ such that $C = D(nc)$.
 \end{proof}

The following lemmas provide further geometric obstructions to isomorphisms.  The term ``level" will be used as defined for elements after Lemma \ref{normal form} and as defined for orbitals after Corollary \ref{levelorb}. \\

\begin{lem}\mylabel{level iso}
Let $\phi: G \longrightarrow H$ be an isomorphism between solvable groups where $G$ is generated by the signatures of the pure fundamental tower $T = \{(A_1, f_1), \cdots, (A_n, f_n)\}$ and $H$ is generated by the signatures of the pure  fundamental tower $S = \{(B_1,g_1), \cdots, (B_n, g_n)\}$.  For each $1 \leq i \leq n$ and for every bump $b$ of $\phi(f_i)$, level$_G(f_i)$ = level$_H(b)$  where the first level function is with respect to the generators of $G$ and the second is with respect to the generators of $H$. 
\end{lem}

\begin{proof}
By Proposition \ref{puretower}, there exists an induced map $\p{T}{M}$ which is an isomorphism of ordered sets.  Let $i \in \overline{n}$ and $(O_1, \phi(f_i)) = \p{T}{M}(A_i, f_i)$. Then there exists a bump $b_1$ of $\phi(f_i)$ with orbital $O_1$, and level($O_1$) = $i$ because $\p{T}{M}$ is an isomorphism of ordered sets.  By Lemma \ref{level3}, level$_H(b_1) = i$. If $f_i$ is a one-bump function, then we are done.  If not, we show that all bumps of $\phi(f_i)$ are at level $i$.  

Let $(O_2, \phi(f_i))$ be any other signed orbital of $\phi(f_i)$ and suppose $b_2$ is the bump of $\phi(f_i)$ with orbital $O_2$. Recall that every bump of a function in $H$ is also in $H$ thanks to \ref{allbumps}.  Suppose level($b_2$) = $j \neq i$. Assume $i < j$.  The proof for $j < i$ is similar.  By Lemma \ref{level3}, level($O_2$) = $j$.   The orbitals $O_1$ and $O_2$ are contained in maximal stacks $S_1$ and $S_2$ of $H$.  By \ref{max iso}, these stacks are conjugate, that is, there is a $c \in H$ such that $S_1c=S_2$.  Since $i < j$ and conjugation induces an order isomorphism on stacks, $O_1c \subset O_2$.  Because $H$ has no transition chains, $c$ has an orbital $C$ containing both $O_1$ and $O_2$.  Let $c' = c|_{C}$ and note that $O_1c' \subset O_2$.  Thus, $R:=[\phi(f_i)^{c'}, \phi(f_i)] \neq 1$ on $O_2$ due to the presence of the two-tower $U = \{(O_1c', \phi(f_i)^{c'}), (O_2, \phi(f_i))\}$.  Since $\phi$ is an isomorphism, $\phi^{-1}(R) = [f_i^{\phi^{-1}(c')}, f_i] \neq 1$ in $G$.  

Since conjugation induces an order isomorphism of stacks, the level of $A_i$ equals the level of $A_i\phi^{-1}(c')$.  Thus $A_i$ and $A_i\phi^{-1}(c')$ do not form a stack.  Since $G$ has no transitions chains, the only remaining option is that $f_i^{\phi^{-1}(c')}$ and $f_i$ share the orbital $A_i$ because they are pure and do not commute. The tower $U$ is untwisted because $c'$ is a one-bump function.  Thus $\phi^{-1}(U)$ is untwisted by Lemma \ref{untwisted image}.  By \ref{OrderUntwist}, there exists a $2$-tower in $\phi^{-1}(U)$, and we have reached a contradiction.  Therefore, $i = j$.

\end{proof}

\begin{cor}\mylabel{level iso cor}
Let $\phi: G \longrightarrow H$ be an isomorphism between solvable groups where $G$ is generated by the signatures of the pure fundamental tower $T = \{(A_1, f_1), \cdots, (A_n, f_n)\}$, and $H$ is generated by the signatures of the pure  fundamental tower $S = \{(B_1,g_1), \cdots, (B_n, g_n)\}$. Then 
 \begin{enumerate}[label={(\arabic*)},ref={\thecor~(\arabic*)}]
 \item level$_G(f_i)$ = level$_H(\phi(f_i))$. 
 \item level$_H(f_i)$ = level$_G(f_i^c)$ = level$_H(\phi(f_i^c))$ for all $c \in G$.  Furthermore, all bumps of $\phi(f_i^c)$ have the same level in $H$.
 \end{enumerate}
\end{cor}

\begin{proof}
The first point follows because of the definition of level. The second follows because conjugation induces an order isomorphism on maximal towers.
\end{proof}

\subsection{Inj-isomorphisms and Direct Systems}

In this section, we complete the proof of Theorem \ref{main}.  That is, we show that if $C$ and $D$ are two countable ordered sets and $W_C$ and $W_D$ are groups as constructed in \ref{construction}, then $W_C \cong W_D$ if and only if $C$ and $D$ are order isomorphic.

 We think of each of the groups and their corresponding ordered sets as direct limits.  To this end, let $C$ be a countable ordered set and $W_C$ be the associated group constructed in \ref{construction}.  Number the elements of $C$ so $C = \{x_1, x_2, \cdots\}$.  Let $g_{x_i}$ be the element of $W_C$ associated to $x_i$, define $X_i = \{x_1, x_2, \cdots, x_i\}$, and define $G_i = <g_{x_1}, g_{x_2}, \cdots, g_{x_i}>$ for each $i \in \mathbb{Z}^+$.  Then, $C$ is the direct limit of the direct sequence $X_1 \longrightarrow X_2 \longrightarrow \cdots$ where the bonding maps are inclusions.  Similarly, $W_C$ is the direct limit of the direct sequence $G_1 \longrightarrow G_2 \longrightarrow \cdots$ where the bonding homomorphisms are also inclusions.  Note that in general the order of these subscripts will not reflect the total order on $C$.  

We will consider an isomorphism between 2 groups represented as direct limits, so we establish more notation and diagrams.  Let $D$ be a countable ordered set and $W_D$ be the associated group.  Number the elements of $D$ so $D = \{y_1, y_2, \cdots\}$.  Let $h_{y_i}$ be the element of $W_D$ associated to $y_i$, define $Y_i = \{y_1, y_2, \cdots, y_i\}$, and define $H_i = <h_{y_1}, h_{y_2}, \cdots, h_{y_n}>$ for each $i \in \mathbb{Z}^+$.  Then, $D$ is the direct limit of the direct sequence $Y_1 \longrightarrow Y_2 \longrightarrow \cdots$ where the bonding maps are inclusions.  Also, $W_C$ is the direct limit of the direct sequence $H_1 \longrightarrow H_2 \longrightarrow \cdots$ where the bonding homomorphisms are inclusions.  

The groups $W_C$ and $W_D$ are generated by pure towers, hence have no transition chains.  By the construction in \ref{construction}, the groups $G_i$ and $H_i$ are generated by pure finite towers $T_i = \{(A_{1i}, g_{1i}), \cdots, (A_{n_ii}, g_{n_ii})\}$ and $S_i =  \{(B_{1i}, h_{1i}), \cdots, (B_{m_ii}, h_{m_ii})\}$, respectively, for each $i \in \mathbb{Z}^+$.  Note we changed the subscripts here.  The first value of the subscript reflects the order in the generating tower, and the second subscript indicates the group (either $G_i$ or $H_i$) in which the generator resides.  We use this notation henceforth.

We will show that given a inj-isomorphism $P$ between the direct systems $D_1$ and $D_2$ for $W_C$ and $W_D$, we can construct a inj-isomorphism $P'$ between the direct systems for $C$ and $D$.  We always change the direct systems $D_1$ and $D_2$ to suit the inj-isomorphism by composing maps and renumbering subscripts so that the arrows in $P$ do not skip any subscripts.  Thus, given a commutative diagram

\begin{equation}\mylabel{system}
\begin{tabular}{>{$}c<{$} >{$}c<{$} >{$}c<{$} >{$}c<{$} >{$}c<{$} >{$}c<{$} >{$}c<{$}}
G_1 &		\longrightarrow & G_2  &\longrightarrow & G_3 & \longrightarrow & \cdots \\
\downarrow & 	\nearrow &	\downarrow & \nearrow & \downarrow & \nearrow & \cdots \\
H_1 & 		\longrightarrow &  H_2 & \longrightarrow & H_3 &  \longrightarrow & \cdots
\end{tabular}
\end{equation}
we will construct a commutative diagram

\begin{equation}\mylabel{systemset}
\begin{tabular}{>{$}c<{$} >{$}c<{$} >{$}c<{$} >{$}c<{$} >{$}c<{$} >{$}c<{$} >{$}c<{$}}
X_1 &		\longrightarrow & X_2  &\longrightarrow & X_3 & \longrightarrow & \cdots \\
\downarrow & 	\nearrow &	\downarrow & \nearrow & \downarrow & \nearrow & \cdots \\
Y_1 & 		\longrightarrow &  Y_2 & \longrightarrow & Y_3 &  \longrightarrow & \cdots
\end{tabular}
\end{equation}

In each diagram above, the middle row of arrows gives the maps which constitute the inj-isomorphism.  For a general square in each of the above diagrams, we name the maps as detailed below:

\begin{equation}\mylabel{square}
\begin{tikzcd}
G_i \arrow[r, red, "\iota_i"]  \arrow[d, blue, "d_i"'] & G_{i+1} \arrow[d, blue, "d_{i+1}"] \\ 
 H_i\arrow[ru, blue, "u_i"]  \arrow[r, red, "j_i"']
&H_{i+1}
 \end{tikzcd}
 \end{equation}

 \begin{equation}\mylabel{squareset}
\begin{tikzcd}
X_i \arrow[r, red, "\iota_i'"]  \arrow[d, blue, "d_i'"'] & X_{i+1} \arrow[d, blue, "d_{i+1}'"] \\ 
 Y_i\arrow[ru, blue, "u_i'"]  \arrow[r, red, "j_i'"']
&Y_{i+1}
 \end{tikzcd}
 \end{equation}

For each $i \in \mathbb{Z}^+$ , the maps $\iota_i, j_i$ and $\iota_i', j_i'$ are the bonding inclusions.  The maps $d_i, u_i$ are from a given inj-isomorphism $P$ between the direct sequences of the groups $W_C$ and $W_D$.  The maps $d_i'$ and $u_i'$ will be implied later as a part of a inj-isomorphism $P''$ that we construct on direct systems of ordered sets which are isomorphic to those in $P'$.  

The following fact about inj-isomorphisms is standard.

\begin{lem}
 Suppose $G$ and $H$ are direct limits of the direct systems $D_1$ and $D_2$, respectively.
 \begin{enumerate}[label={(\arabic*)},ref={\thecor~(\arabic*)}]
  \item If the groups in $D_1$ and $D_2$ are finitely generated, then $G$ and $H$ are isomorphic if and only if there exists a inj-isomorphism between $D_1$ and $D_2$.  An analogous result holds for direct systems of finite ordered sets.
\item If $D_1$ and $D_2$ are injective, meaning the maps in the systems are, then any inj-isomorphism $P$ between them consists entirely of injective maps.
 \end{enumerate}
\end{lem}

Before we develop more tools, note the statement that $W_C \cong W_D$ if and only if $C$ is order isomorphic to $D$ is much stronger than requiring mutual embedding of $C$ and $D$.  There is no Cantor-Schroeder-Bernstein theorem for ordered sets.  For example, the rationals $\mathbb{Q}$ and the rationals adjoined by a maximal element $\mathbb{Q} + *$ mutually embed in each other.  However, they are not order isomorphic since the latter contains a maximal element while the former does not.  For a characterization of countable ordered sets into bi-embeddability classes, see \cite{orderedsets}.   

In the following proposition, we include a subscript on the level to indicate the group in which the level is being measured.  

\begin{prop}\mylabel{level embed}
 Let $W_C, W_D \in WC$ such that $D_1, D_2$ are direct systems for $W_C,W_D$ and let $P$ be a inj-isomorphism between $D_1$ and $D_2$ as shown in diagram (\ref{system}) with a single square illustrated in (\ref{square}).  If $d_i: G_i \longrightarrow H_i$ is any map in $P$ and $g_{pi}$ is any signature in the generating tower $T_i$ of $G_i$, then level$_{H_i}(b_1)$ = level$_{H_i}(b_2)$ for all bumps $b_1, b_2$ of $d_i(g_{pi})$.   There is a parallel statement that applies to each $u_i$.
\end{prop}

\begin{proof}
This proof has similarities with the proof of \ref{level iso}.  However, it requires a bit more subtlety with the mappings and we will also utilize commutativity of the left triangle in (\ref{square}).  The notation will be as developed prior to the proposition.

By Proposition \ref{puretower}, there exists a maximal induced map $\hat{d_i}_{(T_i, M)}$ which is an isomorphism of ordered sets.  Let $(O_1, \phi(g_{pi})) = \hat{d_i}_{(T_i, M)}(A_{pi}, g_{pi})$ where $p \in \overline{n_i}$. Then there is a bump $b_1$ of $d_1(g_{pi})$ which has orbital $O_1$ and level$_{d_i(G_i)}(b_1) = p$ because $\hat{d_i}_{(T_i, M)}$ is an isomorphism of ordered sets and level$_{G_i}(g_{pi})$ = $p$.  However, $H_i$ could be a larger group than $d_i(G_i)$, so level$_{H_i}(b_1)$ could be some other number, say $j$. If $d_i(g_{pi})$ is a one-bump function, then we are done.  If not, we show that all bumps of $d_i(g_{pi})$ are at level $j$.  

Let $(O_2, d_i(g_{pi}))$ be any other signed orbital of $d_i(g_{pi})$ and let $b_2$ be the bump of $g_{pi}$ which has orbital $O_2$.   Recall that every bump of a function in $H_i$ is also in $H_i$ thanks to \ref{allbumps}.  Suppose level$_{H_i}(b_2) = k \neq j$ and $j < k$.  The proof when $k < j$ is similar. By \ref{level3}, level$_{H_i}(O_2) = k$.  The orbitals $O_1$ and $O_2$ are contained in maximal stacks $M_1$ and $M_2$ of $H_i$.  By \ref{max iso}, these stacks are conjugate, that is, there is a $c \in H$ such that $M_1c=M_2$.  Since $j < k$ and conjugation induces an order isomorphism on stacks, $O_1c \subset O_2$.  Because $H_i$ has no transition chains, $c$ has an orbital $C$ containing both $O_1$ and $O_2$.  Let $c' = c|_{C}$ and note that $O_1c' \subset O_2$.  Thus, $R:=[(d_i(g_{pi}))^{c'}, d_i(g_{pi})] \neq 1$ on $O_2$ due to the presence of the two-tower $U = \{(O_1c', (d_i(g_{pi}))^{c'}), (O_2, d_i(g_{pi})\}$.  Note that unlike the proof of \ref{level iso} where the conjugator is in the image of the map $\phi$, the conjugator $c'$ is not necessarily in the image of $d_i$.  However, since $u_i$ is an injective homomorphism, $u_i(R) \neq 1$ in $G_{i+1}$.

Consider the 2-tower $V := \{(O_2, d_i(g_{pi})), (C,c')\}$.  The image $u_i(V)$ contains a 2-tower $W$ thanks to \ref{OrderUntwist}.  By assumption, $\iota_i(g_{pi}) = g_{l, i+1}$ for some signature $g_{l, i+1}$ in the generating tower of $G_{i+1}$, and by commutativity $u_i(d_i(g_{pi})) = \iota_i(g_{pi}) = g_{l, i+1}$.  Since $g_{l, i+1}$ is a one-bump function and order is preserved by $\hat{u_i}_{(V,W)}$, $W$ is the unique 2-tower in $u_i(V)$ and $W = \{(A_{l, i+1}, g_{l, i+1}), (C', c'')\}$ where $c'' = u_i(c')$ and $C'$ is the orbital of $c''$ containing $A_{l, i+1}$.  Therefore, $u_i(R) = [g_{l, i+1}^{c''}, g_{l, i+1}] = 1$, a contradiction.
\end{proof}

\begin{cor}
Every maximal induced map $\hat{d_i}_{(T_i, -)}$ induces the same map on levels.  Furthermore, for all $c \in G$ the maximal induced map $\hat{d_i}_{(T_i^c,-)}$ induces the same map on levels as $\hat{d_i}_{(T_i, -)}$. A parallel statement applies to each $u_i$.
\end{cor}

In the following, we assume each level is measured in the ambient group of the direct system in which the element lives.  Therefore, we omit subscripts of the level to simplify notation.

\begin{lem}\mylabel{level gen}
Given the same setup as Proposition \ref{level embed}, refer to diagram (\ref{square}).  If level$(d_i(g_{ki}))$ = $l$ and level$(\iota_i(g_{ki}))$ = $m$, then level$(u_i(h_{li}))$ = $m$.  (Similarly, if level$(u_i(h_{ki}))$ = $l$ and level$(j_i(h_{ki}))$ = $m$, then level$(d_{i+1}(h_{li}))$ = $m$.)
\end{lem}

\begin{proof}
 Assume level$(d_i(g_{ki}))$ = $l$ and level$(\iota_i(g_{ki}))$ = $m$. By Lemma \ref{level4}, level$(d_i(g_{ki}))$ = max$\{\text{level}(O) \, | \, \text{$O$ is an orbital of $d_i(g_{ki})$}\}$.  Thus there exists an orbital $B$ of $d_i(g_{ki})$ which has level $l$ in $H_i$.  By Proposition \ref{level embed}, all orbitals of $d_i(g_{ki})$ have level $l$ in $H_i$.  Let $b$ be the bump of $d_i(g_{ki})$ with orbital $B$, and note level$(b) = l$.  Furthermore, our direct systems were constructed to map generating tower signatures to generating tower signatures.  Since level$(\iota_i(g_{ki}))$ = $m$, we must have $\iota_i(g_{ki}) = g_{m, i+1}$.  By commutativity, $u_i(d_i(g_{ki})) = \iota_i(g_{ki}) = g_{m, i+1}$.

The element $h_{li}$ is a one-bump generator at level $l$ in $H_i$.  Since maximal stacks are conjugate in $H_i$, there is a $c \in H_i$ such that $h_{li}^c$ shares the orbital $B$ with $d_i(g_{ki})$.  Either $[h_{li}^c, d_i(g_{ki})] = 1$ or not.  Since $h_{li}^c$ is a one-bump function, this statement is equivalent to a relative one: Either $[h_{li}^c, d_i(g_{ki})] = 1$ on $B$ or not.  

Suppose $l \neq 1$. If $[h_{li}^{c}, d_i(g_{ki})] = 1$, we argue we can alter the conjugator $c$ to some $c'$ to guarantee that $[h_{li}^{c'}, d_i(g_{ki})] \neq 1$ on $B$.   Let $c' = (h_{l-1, i})c$. By \ref{multiply}, $h^{c'}$ also has orbital $B$.  Since $H_i$ has no transition chains, both $c$ and $c'$ have some orbital $C$ containing the orbital $B_{li}$ of $h_{li}$ and the orbital $B$ of $d_i(g_{ki})$. 

We now argue $[h_{li}^{c'}, d_i(g_{ki})] \neq 1$.  If not, both $[h_{li}^c, d_i(g_{ki})] = [ h_{li}^{c'}, d_i(g_{ki})] = 1$ on $B$.  Footnote * of Table 2 implies the groups $K_1:= <h_{li}^c, d_i(g_{ki})>_B$ and $K_2 := <h_{li}^{c'}, d_i(g_{ki})>_B$ are each isomorphic to $\mathbb{Z}$.  Then for some $q_1, q_2 \in \mathbb{Z}$, we have $d_i(g_{ki}) = k_1^{q_1} = k_2^{q_2}$ where $k_1$ is the generator of $K_1$ and $k_2$ is the generator of $K_2$.  Therefore, $k_2^{q_2} \in K_1$, so $[k_2^{q_2}, k_1] = 1$.  By \ref{z}, $[k_2, k_1] = 1$ which implies $[h_{li}^c, h_{li}^{c'}] = 1$.  However, $[h_{li}^c, h_{li}^{c'}] = [h_{li}^c, (h_{li}^{h_{l-1,i}})^c] = [h_{li}, (h_{li}^{h_{l-1,i}})]^c = 1 \Leftrightarrow [h_{li}, h_{li}^{h_{l-1,i}}] = 1$.  Since $h_{li}$ and $h_{li}^{h_{l-1,i}}$ are conjugates, they have the same leading and trailing slopes on corresponding orbitals.  Since $Bh_{l-1,i} = B$, the slopes of $h_{li}$ and $h_{li}^{h_{l-1,i}}$ are the same on the end of $B$.  However, $h_{li} \neq h_{li}^{h_{l-1, i}}$ in $H_i$ since their normal forms are distinct.  Thus, they are equal on the ends of $B$ but not throughout, so they do not commute.  Therefore, $[h_{li}^c, h_{li}^{c'}] \neq 1$, a contradiction.  

The element $u_id_i(g_{ki}) = \iota_i(g_{ki})$, hence is a one-bump function due to how the direct systems of groups were constructed.  Therefore, $\mathcal{SO}(u_i(\{d_i(g_{ki}), h_{li}^{c'}\}))$ is untwisted.  By \ref{nontower}, there exist orbitals $O_1, O_2$ of $u_id_i(g_{ki}), u_i(h_{li}^{c'})$ such that $O_1 = O_2 (nc)$.  Thus $m$ = level$(O_1)$ = level$(O_2)$.  Conjugating by $u_i(c'^{-1})$ yields $u_i(h_{li})$.  Since conjugation preserves order, level$(O_2)$ = $m$.  By \ref{level embed}, all bumps of $u_i(h_{li})$ have level $m$, hence level$u_i(h_{li}))$ = $m$.

For the $l=1$ case, note all orbitals of $d_i(g_{ki})$ have level 1 in $H_i$ due to \ref{level embed}.  Thus all orbitals of $d_i(g_{ki})$ are minimal in $\mathcal{O}(H_i)$.  We conclude $k=1$ because otherwise $d_i(g_{k-1, i})$ would exist and would have an orbital contained in an orbital of $d_i(g_{ki})$.  Therefore, $g_{ki} = g_{1i}$ and $h_{li} = h_{1i}$.  Applying normal form, we obtain, $d_i(g_{ki}) = d_i(g_{1i}) = w_1^{c_1}w_2^{c_2} \cdots w_r^{c_r}$ where for each $1 \leq j \leq r$, $w_j \in <h_{1i}>$ and $c_j \in <h_{2i}, h_{3i}, \cdots, h_{m_i}>$.  Then, $u_id_i(g_{1i}) = w_1'^{c_1'}w_2'^{c_2'} \cdots w_r'^{c_r'}$ where for each $1 \leq j \leq r$, $w_j' \in <u_i(h_{1i})>$ and $c_j' = u_i(c_j)$.  Therefore, level$u_id_i(g_{1i}))$ = level$(u_i(h_{1i}))$.  By commutativity, $u_id_i(g_{1i}) = \iota_i(g_{1i}) = g_{m, i+1}$.  Thus, level$u_i(h_{1i}))$ = $m$.
 \end{proof}

\begin{prop}\mylabel{set inj-iso}
 Under the assumptions of \ref{level embed}, if two direct systems of groups are inj-isomorphic via $P$, then $P$ induces a inj-isomorphism $P'$ of the direct systems of ordered sets of generators of each group. 
\end{prop}

\begin{proof}
We refer to diagrams (\ref{system}), (\ref{systemset}), (\ref{square}), and (\ref{squareset}) and the setup of the direct systems described there. In particular, we focus on a general square

\begin{center}
\begin{tikzcd}
X_i \arrow[r, red, "\iota_i'"]  \arrow[d, blue, "d_i'"'] & X_{i+1} \arrow[d, blue, "d_{i+1}'"] \\ 
 Y_i\arrow[ru, blue, "u_i'"]  \arrow[r, red, "j_i'"']
&Y_{i+1}
 \end{tikzcd}
 \end{center}
 
Recall $\iota_i'$ and $j_i'$ were defined for each $i \in \mathbb{Z}^+$ and what remained to complete the inj-isomorphism of ordered sets was defining order preserving maps $u_i'$ and $d_i'$, and showing the resulting triangles commute for each $i$.
 
For each $i \in \mathbb{Z}^+$, we define a few sets and several maps.  Let $L_i$ be the set of consecutive natural numbers starting with $1$ that is order isomorphic to the set of points $X_i \subset [0,1]$.  Let $\phi_i: L_i\longrightarrow X_i$ be the corresponding isomorphism.  Similarly, let $L_i'$ be the set of consecutive natural numbers starting with $1$ that is isomorphic to $Y_i$, and let $\psi_i: L_i' \longrightarrow Y_i$ be the corresponding isomorphism. Note the sets $L_i$ and $L_i'$ are the set of levels of orbitals in $G_i$ and $H_i$.  

Define the maps
\begin{align*}
\iota_i'': L_i \longrightarrow L_{i+1} \text{ by } \iota_i''(k) &= \phi_{i+1}^{-1}\iota_i'\phi_i(k),\\
j_i'': L_i' \longrightarrow L_{i+1}' \text{ by } j_i''(k) &= \psi_{i+1}^{-1}j_i'\psi_i(k), \\
 d_i'': L_i \longrightarrow L_i' \text{ by } d_i''(k) &= \text{level}(d_i(g_{ki})), \text{ and }\\
u_i'': L_i' \longrightarrow L_{i+1} \text{ by } u_i''(k) &= \text{level}(u_i(h_{ki})).
\end{align*} 

Note that $\iota_i''$ and $j_i''$ are not necessarily inclusions.  Let $P''$ be the collection of sets and maps just defined.  We claim $P''$ is a inj-isomorphism of ordered sets.  For each $i \in \mathbb{Z}^+$, we show a square

\begin{center}
\begin{tikzcd}
L_i \arrow[r, red, "\iota_i''"]  \arrow[d, blue, "d_i''"'] & L_{i+1} \arrow[d, blue, "d_{i+1}''"] \\ 
 L_i'\arrow[ru, blue, "u_i''"]  \arrow[r, red, "j_i''"']
&L_{i+1}'
 \end{tikzcd}
 \end{center}
 in $P''$ is commutative. Note, $\iota_i(g_{ki}) = u_id_i(g_{ki})$ implies level$(\iota_i(g_{ki}))$ = level$(u_id_i(g_{ki}))$.  We wish to show $\iota_i'' = u_i''d_i''$.  But, $u_i''d_i''(k) = u_i''($level$(d_i(g_{ki})))$ = level$u_i(h_{\text{level}(d_i(g_{ki})), i}$) = level($\iota_i(g_{ki})$)  by Lemma \ref{level gen}.  Since level($\iota_i(g_{ki})$) = $\iota_i''(k)$, our proof of commutativity is complete for the first triangle.  The proof for the other triangle is similar.  
 
If $m < n$ in $L_i$, then $g_{mi}$ has orbital properly contained in the orbital of $g_{ni}$ in the group $G_i$.  Since $d_i$ is an injective homomorphism, Scholium \ref{OrderUntwist} implies there exists an orbital $O_m$ of $d_i(g_{mi})$ properly contained in an orbital $O_n$ of $d_i(g_{ni})$. Thus, level$(O_m) < $ level$(O_n)$.  Proposition \ref{level embed} shows that for each generator $g_{pi}$ in the pure generating tower of $G_i$, the bumps and hence orbitals of $d_i(g_{pi})$ are all at the same level.  Hence the maximum level of an orbital of $d_i(g_{pi})$ is equal to the level of any of its orbitals.   Therefore $d_i''(m) = $ level$(O_m) < $ level$(O_n) = d_i''(n)$, so $d_i''$ preserves order.  The same kind of argument applies to any map in $P''$.
 
The maps $u_i'$ and $d_i'$ in (\ref{squareset}) are easily obtained by composing with the appropriate isomorphisms between the sets $X_i, Y_i$ and level sets $L_i, L_i'$.
\end{proof}

The next corollary gives the last part of Theorem \ref{main}.

\begin{cor}
Groups in $WC$ are isomorphic if and only if their generating towers are isomorphic as ordered sets.
\end{cor}

\end{document}